\PassOptionsToPackage{dvipsnames}{xcolor}
\PassOptionsToPackage{
  colorlinks=true,
  linkcolor=RoyalPurple,
  citecolor=Blue,
  urlcolor=Violet
}{hyperref}

\documentclass[11pt]{article}
\usepackage{mysettings}
\usepackage{subfiles}
\numberwithin{equation}{section}
\allowdisplaybreaks 

\usepackage{amsmath}
\usepackage{multirow}
\usepackage{graphicx}
\usepackage{caption}
\usepackage{subcaption}
\usepackage{bm}
\usepackage{footnote}
\usepackage{stfloats}
\usepackage{placeins}
\usepackage{color,soul}
\usepackage{dsfont}
\usepackage{amssymb}
\usepackage{pbox}
\usepackage{etextools}
\usepackage{mathtools}
\usepackage{tabulary}
\usepackage{dsfont}
\usepackage{svg}

\usepackage{todonotes}

\usepackage{tikz}
\usetikzlibrary{calc} 

\usepackage{comment}

\usepackage{wrapfig}

\usepackage{bm}

\newcommand{\cM}{\mathcal{M}}

\newcommand{\cT}{{\mathcal{T}}}

\newcommand{\cX}{\mathcal{X}}

\newcommand{\RR}{{\mathbb R}}

\newcommand{\reals}{{\mathbb R}}

\newcommand{\dist}{\operatorname{dist}}

\newcommand{\dom}{\operatorname{dom}}

\def\shortdisplay{\setlength{\abovedisplayskip}{5pt}%
	\setlength{\belowdisplayskip}{5pt}%
	\setlength{\abovedisplayshortskip}{2pt}%
	\setlength{\belowdisplayshortskip}{2pt}}

\let\oldselectfont\selectfont
\def\selectfont{\oldselectfont\shortdisplay}

\newcommand{\symMat}{\mathbb{S}}

\newcommand{\norm}[1]{\left\lVert#1\right\rVert}
\newcommand{\matr}[1]{{\bm{#1}}}

\newcommand{\ip}[1]{\langle #1 \rangle}

\newcommand{\argmin}{\operatorname{argmin}}

\newcommand{\CVec}{C_{T}} 
\newcommand{\CRet}{C_{R}} 
 
\newcommand{\fLip}{G_f}

\newcommand{\twoCut}[1]{\tilde{f}_{#1}} 
\newcommand{\model}[1]{\widehat{f}_{#1}} 
\newcommand{\proxGap}[1]{\Delta_{#1}} 
\newcommand{\modProxGap}[1]{\tilde{\Delta}_{#1}^{(\rho_{#1})}} 

\newcommand{\VecT}[2]{\mathcal{T}_{#1 \hookleftarrow #2}} 
\newcommand{\ParT}[2]{\mathcal{P}_{#1 \hookleftarrow #2}} 
\newcommand{\vecTObj}[1]{{\hat{g}}_{#1}}

\newcommand{\modShift}[1]{\kappa_{#1}}  
\newcommand{\genShift}[2]{\kappa(#1, #2)}

\newcommand{\stepsize}[1]{\rho_{#1}} 
\newcommand{\manDist}[2]{d_{\mathcal{M}}\left(#1,#2 \right)} 
\newcommand{\canDir}[1]{v_{#1}} 
\newcommand{\twoCutDir}{\tilde{v}_{t+2}} 
\newcommand{\candR}[1]{r_{\operatorname{dir}}(\rho_{#1})} 
 
\newcommand{\trueProxStep}{\bar{v}}
\newcommand{\constModelLip}[1]{h_{\operatorname{b}}(#1)} 
\newcommand{\swapGap}{\delta_{\operatorname{last}}} 
\newcommand{\noDoubleConst}{A}

\newcommand{\minimizers}{\mathcal{X}_{\ast}}

\usepackage{nicefrac} 
\usepackage{microtype} 
\usepackage[dvipsnames]{xcolor}
\title{Nonsmooth Riemannian optimization with inexact manifold primitives via bundle methods}

\author{
    Mateo D\'{i}az\thanks{Department of Applied Mathematics and Statistics, Johns Hopkins University, Baltimore, MD 21218, USA.}
    \and Benjamin Grimmer$^{\ast}$
    \and Ian McPherson$^{\ast}$
}

\begin{document}
\maketitle
\renewcommand{\thefootnote}{\relax}
\footnotetext{\textbf{Funding:} MD was partially supported by NSF awards CCF 2442614 and DMS 2502377. MD and BG were supported as fellows of the Alfred P.~Sloan Foundation.}
\renewcommand{\thefootnote}{\arabic{footnote}}

\begin{abstract}
	Optimization on Hadamard manifolds---the natural Riemannian setting for globally geodesically convex problems---relies on exponential maps to retract tangent vectors and parallel transport to connect tangent spaces across the manifold. These primitives are often computationally expensive, leading software packages to rely on approximations: first-order retractions and vector transports. However, existing results for optimization on Hadamard manifolds either require exact primitives or lack non-asymptotic rates. We bridge this gap by introducing a proximal bundle method for nonsmooth geodesically convex optimization and establishing the first oracle-complexity bounds that rely only on subgradients and inexact primitives. We obtain sublinear rates for general objectives and optimal linear convergence under sharp function growth.


\end{abstract}

\section{Introduction}
We consider optimization problems of the form \begin{equation*}
	\min_{x \in \mathcal{M}} f(x),
\end{equation*}
where $\mathcal{M}$ is a Hadamard manifold and $f \colon \mathcal{M} \rightarrow \mathbb{R}$ is a nonsmooth geodesically convex ($g$-convex) function. A Hadamard manifold is a complete and simply connected Riemannian manifold with nonpositive sectional curvature. A key feature of Hadamard manifolds is that they exhibit unique geodesics between any pair of points, which causes several natural objects to behave as in Euclidean spaces. For instance, the squared distance function is $g$-convex---a fundamental property for defining proximal operators---whereas it fails to be $g$-convex on general Riemannian manifolds. This favorable geometry arises in a broad range of applications, including covariance matrix estimation \cite{wiesel2012geodesic}, representation learning \cite{nickel2017poincare}, diffusion tensor imaging \cite{pennec2006riemannian}, and radar signal processing \cite{arnaudon2013riemannian}.


Algorithms for Hadamard optimization rely on two primitives: exponential maps, which retract tangent vectors onto the manifold, and parallel transports, which translate tangent vectors between tangent spaces. Assuming access to these primitives, \cite{zhang2016firstordermethodsgeodesicallyconvex} and \cite{bento2017iteration} established the first complexity bounds for the subgradient method and the proximal point method, respectively. Nonetheless, exponential maps and parallel transport lack closed-form expressions on most manifolds of interest. 
Indeed, it was recognized early by \cite{zhang2016firstordermethodsgeodesicallyconvex} that
\vspace*{5pt}
\begin{center}
	\textit{``(...) it is often favorable to replace exponential mapping with computationally cheap retractions, it is important to understand the effect of this approximation on convergence rate. Analyzing this effect is of both theoretical and practical interests.''}
\end{center}
\vspace*{5pt}
Software packages often replace exponential maps and parallel transports with cheaper approximations, namely first-order retractions and transporters \cite{boumal2023intromanifolds}. For smooth objectives, \cite{boumal2019global} proved complexity guarantees for gradient descent to reach an approximate stationary point under general retractions. Their analysis, however, relies on Lipschitz continuity of the gradient and does not carry over to the nonsmooth setting. This naturally leads to the main question studied in this work.
\begin{tcolorbox}[myboxstyle] \centering
	\textit{
		Are there provably convergent algorithms for nonsmooth $g$-convex optimization using only subgradients, first-order retractions, and transporters?
	}
\end{tcolorbox}

We answer this question affirmatively by extending Euclidean proximal bundle methods to the Hadamard setting. These classical algorithms are widely used in practice for three main reasons: \((i)\) unlike subgradient methods, they guarantee function value descent; \((ii)\) they naturally adapt to favorable growth conditions and can converge faster when such structure is present; and \((iii)\) they are robust to parameter misspecification \cite{díaz2023optimal}.

\subsection{Contributions}
Our goal throughout is to find an $\varepsilon$-minimizer, i.e., a point $x \in \cM$ satisfying $f(x) - \min_{y \in \cM} f(y) \leq \varepsilon.$ We summarize our two main contributions.

\vspace{0.5em}
\begin{enumerate}[leftmargin=.5cm]
  \item[] (\textbf{Rates for $g$-convex objectives})
We introduce a proximal bundle method for Hadamard manifolds and establish convergence rates for nonsmooth $g$-convex objectives using only subgradients, first-order retractions, and transporters. The method requires bounds on the retraction and transporter errors as well as a lower bound on the sectional curvature of \(\cM\); such quantities are available for many manifolds of interest. For Lipschitz objectives, we show an iteration complexity of \(\mathcal{O}(\rho\cdot\varepsilon^{-3})\), where the proximal parameter \(\rho\) is chosen adaptively by the algorithm. Our rates interpolate between flat and negatively curved settings: under zero curvature there is no constraint on \(\rho\), whereas negative curvature forces \(\rho\) to grow with the desired accuracy. When the curvature is small relative to \(\varepsilon\), the algorithm can select \(\rho \propto \varepsilon\), recovering the optimal Euclidean rate \(\mathcal{O}(\varepsilon^{-2})\) of \cite{díaz2023optimal}. Under large curvature, \(\rho\) may grow as \(\varepsilon^{-2}\) in the worst case, yielding an overall rate of \(\mathcal{O}(\varepsilon^{-5})\).
        \footnote{Preliminary numerical experiments indicate the threshold ${\rho} \propto \varepsilon^{-2}$ may be highly pessimistic.}
\vspace{0.5em}
  \item[] (\textbf{Faster rates under H\"{o}lder growth})
        We show that, under \(p\)th-order H\"older growth\footnote{A loss $f$ exhibits $p$th H\"{o}lder growth if $f(x) - \min_\cM f \geq \mu \cdot \dist(x, \cX_\star)^p$ with $\cX_\star = \argmin_{\cM} f$ and $\dist(x, \cX_\star) = \inf_{y \in \cX_\star} d_{\cM}(x,y)$.} with \(p \geq 1\), our method enjoys faster convergence rates when paired with an idealized proximal parameter schedule, which requires knowledge of the objective gap and growth modulus $\mu$.
In the sharp case \(p=1\), we establish an optimal linear rate. For \(p \in (1,4/3]\), we show an \(\mathcal{O}(\varepsilon^{-(2-2/p)})\) rate, which matches the optimal Euclidean complexity (see \Cref{thm: convergence rate with growth tuned}). For \(p \geq 4/3\), we obtain the slower rate \(\mathcal{O}(\varepsilon^{-(5-6/p)})\). We believe that both the threshold \(p=4/3\) and the ensuing suboptimality are proof artifacts rather than fundamental barriers.
\end{enumerate}
\vspace{0.5em}

\noindent As an additional algorithm design contribution, we show the following.

\vspace{0.5em}
\begin{enumerate}[leftmargin=.5cm]
    \setcounter{enumi}{2}
\item[] (\textbf{Bounded memory}) Our Riemannian proximal bundle method provably achieves the above rates while using at most three affine cuts at each iteration, that is, storing only three subgradients. By contrast, prior bundle methods on Riemannian manifolds
required exact full-memory models---whose bundle size grows with the iteration count---to guarantee asymptotic convergence.
\end{enumerate}
\vspace{0.5em}


\subsection{Related Work}


\paragraph{Proximal bundle methods} 
Proximal bundle methods have a long history for Euclidean problems, they were independently introduced by~\cite{Lemarechal1975, Mifflin1977, wolfe1975method}.
The first explicit complexity was derived by 
\cite{kiwiel2000proximal}, who showed a $\mathcal{O}({\rho} \cdot \varepsilon^{-3})$ rate for Lipschitz convex problems, where $\rho$ is the largest proximal parameter used during the algorithm's execution.
We recover this same dependence in the Hadamard setting, although curvature forces ${\rho}$ to grow larger. Under strong convexity, 
\cite{Du2017} showed a faster $\mathcal{O}(\varepsilon^{-1}\log(\varepsilon^{-1}))$ rate. Soon after, 
\cite{liang2020iteration} derived an optimal rates for a variant of the classical bundle method that are optimal up to logarithmic factors for both convex and strongly convex problems. More recently, 
\cite{díaz2023optimal} introduced a general analysis technique showing that bundle methods automatically adapt to smoothness and H\"{o}lder growth conditions. They further showed that when paired with an idealized proximal parameter schedule bundle methods attain minimax optimal rates for the class of Lipschitz functions, with and without H\"{o}lder growth. By adapting their technique, we recover optimal complexities under idealized schedules in certain growth regimes. However, curvature forces the proximal parameter in the Riemannian setting  to be larger than their Euclidean counterparts in some growth regimes, leading to suboptimal rates. 

\paragraph{Hadamard optimization}
Early work extended the subgradient method~\cite{ferreira1998subgradient} and the proximal point method~\cite{ferreira2002proximal} to the Hadamard manifold setting. The seminal work of~\cite{zhang2016firstordermethodsgeodesicallyconvex} established the first complexity guarantees for the subgradient method, recovering the classical $\mathcal{O}(\varepsilon^{-2})$ rate for $g$-convex objectives and the $\mathcal{O}(\varepsilon^{-1})$ rate for strongly $g$-convex objectives. As in these results, we require a lower bound on sectional curvature, but we relax the need for exact exponential maps.
On the other hand,~\cite{huang2021riemannianproximal} showed that the forward--backward algorithm attains an $\mathcal{O}(\varepsilon^{-1})$ rate under general retractions. Their analysis, however, requires the objective to be compatible with the retraction, in the sense that its components must satisfy notions of `retraction convexity' and `retraction smoothness.' In contrast, our work addresses black-box nonsmooth optimization without assuming composite structure or retraction-dependent function properties.
More broadly, and further removed from our setting, optimization in metric spaces of nonpositive curvature has recently emerged as an active direction. In this setting,~\cite{lewis2024horoballssubgradientmethod} recovered the standard $\mathcal{O}(\varepsilon^{-2})$ subgradient complexity using a support-ray oracle, while~\cite{goodwin2024subgradient} obtained the same rate through a more tractable splitting oracle based on Busemann subgradients. Recent work studied stochastic problems~\cite{Goodwin_2026, pischke2026busemannsubgradientmethodsstochastic} and integrated metric-space tools into the Riemannian setting~\cite{ferreira2026subdifferentialcharacterizationbusemannfunctions,millan2026}.


\paragraph{Riemannian optimization beyond Hadamard manifolds}
In this broader geometric setting, the loss of global geodesic convexity rules out guarantees of convergence to global minimizers.
For composite problems, existing operator-splitting methods \cite{chen2020proximalgradient,ADMMRiemannian2025,huang2021riemannianproximal} exploit problem structure to obtain nonasymptotic rates for finding approximate critical points. More recently, \cite{sahinoglu2025finitetime} established finite-time guarantees for stochastic, nonsmooth, and nonconvex problems by showing convergence to \((\varepsilon,\delta)\)-Goldstein stationary points using the same oracle class considered here, namely subgradients and first-order retractions. While their focus is on stationarity, our work instead establishes rates in objective gap.
Proximal bundle methods have also recently been extended to manifolds with bounded sectional curvature \cite{bergmann2025convexbundle,hoseini2023proximal}. However, their guarantees remain purely asymptotic and rely heavily on full-memory models as well as exact manifold primitives, such as exponential maps and parallel transport. By contrast, our method uses inexact primitives and maintains a memory footprint of at most three subgradients. Extending our nonasymptotic, retraction-based analysis to this broader bounded-curvature setting is a natural direction for future work.\\

\noindent \textbf{Outline.} Section 2 reviews the necessary background from Riemannian geometry and Euclidean bundle methods. Section 3 presents our Riemannian proximal bundle method and states its convergence results. Section 4 provides detailed proofs of our main results. Finally, Section 5 presents numerical experiments supporting our theoretical findings.

\section{Preliminaries} In this section, we introduce preliminaries and notation from Riemannian geometry and bundle methods. For thorough introductions, we refer the interested reader to the manuscripts \cite{Absil2012, boumal2023intromanifolds, lee2013introduction, tumanifolds}.

\subsection{Riemannian Geometry} \label{subsec: prelim RG and RO}
Let $\cM$ be a smooth manifold with a Riemannian metric $\ip{\cdot, \cdot} := \{\ip{\cdot, \cdot}_x\}_{x \in \cM}$, i.e., a smoothly varying family of inner products. Let $\ip{\cdot,\cdot}_x : T_x \cM \times T_x \cM \rightarrow \reals$ denote the inner product on tangent space $T_x \cM$, and $\norm{\cdot}_x$ the associated norm. The Riemannian metric induces a unique torsion-free, metric-compatible \textit{Levi-Civita connection} $\nabla$.
\textit{Geodesics} on $\cM$ are smooth curves with zero acceleration with respect to $\nabla$.
The Riemannian distance between two points \(x,y\in\cM\) is 
\begin{equation}\label{eq:dist}
d_{\cM}(x,y)
:=
\inf_\gamma
\int_0^1 \norm{\gamma'(t)}_{\gamma(t)}dt
\end{equation}
where the infimum is taking over the set of piecewise smooth curves $\gamma\colon [0,1] \to \cM$ with $\gamma(0) = x$ and $\gamma(1) =x.$ For any $x \in \mathcal{M}$, we denote $B_{\mathcal{M}}(x, \zeta) := \{y \in \cM \colon d_{\cM}(x,y) \leq \zeta\}$ and $B_{x}(0,\zeta) := \{v \in T_{x}\cM \colon \|v\|_{x} \leq \zeta\}$. 

For \(x\in\cM\), the \textit{exponential map} \(\exp_x\) is defined on a suitable
neighborhood \(V_x \subseteq T_x\cM\) of the origin by
\(
\exp_x(v) = \gamma_v(1),
\)
where \(\gamma_v\) is the unique geodesic satisfying \(\gamma_v(0)=x\) and
\(\gamma_v'(0)=v\). The map \(\exp_x\) is smooth and admits a local inverse,
called the \textit{logarithmic map}, denoted by \(\log_x : U_x \to T_x\cM\) for
a suitable neighborhood \(U_x\subseteq \cM\) of \(x\). We say that \(\cM\) is
\textit{(geodesically) complete} if \(\exp_x\) is defined on all of \(T_x\cM\) for
every \(x\in\cM\). When $\cM$ is complete, the Hopf--Rinow theorem states that a minimizing geodesic realizes the infimum in \eqref{eq:dist}.


For \(x \in \cM\) and linearly independent tangent vectors \(v, w \in T_x\cM\), the \textit{sectional curvature} is 
\[
    K_x(v, w) := \frac{\ip{\mathcal{R}_x(v, w)w, v}_x}{\norm{v}_x^2 \norm{w}_x^2 - \ip{v, w}_x^2}, \qquad \mathcal{R}(u, v)w := \nabla_u\nabla_v w - \nabla_v\nabla_u w - \nabla_{[u, v]}w.
\]
here we slightly abuse notation and use $[u, v]$ to denote the Lie bracket for the vector fields $u$ and $v$. Geometrically, \(K_x(v, w)\) measures the curvature of \(\cM\) along the two-dimensional subspace spanned by \(v\) and \(w\), quantifying whether nearby geodesics emanating from \(x\) tend to diverge or converge.
A \textit{Hadamard manifold} is a complete, simply connected Riemannian manifold with nonpositive sectional curvature. On such a manifold, the Cartan--Hadamard theorem guarantees that for any \(x \in \cM\), the exponential map \(\exp_x \colon T_x\cM \to \cM\) is a global diffeomorphism. Thus, both \(\exp_x\) and the logarithmic map \(\log_x \colon \cM \to T_x\cM\) are globally well-defined everywhere on \(\cM\).

Henceforth, we assume that $\cM$ is a Hadamard manifold. Given any $x,y \in \cM$ and their minimizing geodesic $\gamma$, the \textit{parallel transport} $\ParT{y}{x}\colon T_x \cM \rightarrow T_y\cM$ with respect to $\nabla$ 
is the linear operator mapping $v_x \in T_{x} \cM$ to $v_{y} = V(1) \in T_{y}\cM$,
where $V$ is the unique smooth vector field along $\gamma$ satisfying $V(0) = v_{x}$ and $\nabla_{\gamma'}V = 0$.
This map is an \textit{isometry}, and the parallel transport in the opposite direction, $\ParT{x}{y}$, satisfies $\ip{\ParT{y}{x}v_x, w_y}_y = \ip{v_x, \ParT{x}{y}w_y}_x.$ Furthermore, since a geodesic parallel transports its own velocity vector, we have the identity $\ParT{y}{x}\log_x y = -\log_y x.$

A set \(\mathcal{U} \subseteq \cM\) is \textit{geodesically convex} (\(g\)-convex) if for any \(x, y \in \mathcal{U}\), their minimizing geodesic is entirely in \(\mathcal{U}\). Further, a function \(f \colon \cM \to \RR \cup \{\infty\}\) is \textit{geodesically convex} (\(g\)-convex) if for any \(x, y \in \cM\), we have \[f(\gamma(t)) \leq (1-t)f(\gamma(0)) + tf(\gamma(1)) \qquad \text{for all }t \in [0,1],\] where \(\gamma\) is the minimizing geodesic from \(x\) to \(y\).
When $\dom f = \{x \in \cM \mid f(x) < +\infty\}$ is $g$-convex and $f$ is $g$-convex, the \textit{Riemannian subdifferential} of $f$ at $x \in \dom f$ denoted $\partial f(x)$ is given by the vectors $g \in T_x\cM$ satisfying
\[f(y) \geq f(x) + \ip{g_x, \log_x y}_x  \qquad \text{ for all } y \in \dom f.\]
We will use computationally cheaper approximations of the manifold primitives. Concretely, we approximate the exponential map with first-order retractions.
\begin{definition} The map $R_x\colon T_x \cM \rightarrow \cM$ is a \textit{first-order retraction} provided \begin{equation*}
    R_x(0) = x \quad \text{and} \quad DR_x(0) = \operatorname{id}_{T_x \cM}.
\end{equation*}
\end{definition}
A first-order retraction matches zeroth and first-order derivatives of $\exp_x$.
Similarly, we approximate the parallel transport using a weakened notion of transporter.
\begin{definition}[Weak transporter]
  Let $\cM$ be a smooth manifold and let $U \subseteq \cM \times \cM$ be an open set containing the diagonal $\{(x,x) : x \in \cM\}$. A \textit{weak transporter} on $U$ is a mapping that assigns to each $(x,y) \in U$ a continuous map $\VecT{y}{x} \colon T_x\cM \to T_y\cM$ such that $\VecT{x}{x}$ is the identity on $T_x\cM$ for all $x \in \cM$.
\end{definition}
The standard definition of a transporter \cite[Definition~10.61]{boumal2023intromanifolds} additionally requires each $\VecT{y}{x}$ to be linear and the overall mapping $(x, y) \mapsto \VecT{y}{x}$ to be smooth. Our analysis requires neither property. Instead, in the next section we impose that $\VecT{x}{y}$ approximates the parallel transport $\ParT{x}{y}$ in a quantitative sense.

\subsection{Euclidean Bundle Methods}
Proximal bundle methods mimic the proximal point method applied to a nonsmooth convex objective $f \colon \reals^d \rightarrow \reals$ by approximating $f$ with local models built from subgradients at previous iterates---typically a max of linearizations for some set of iterates $J$, meaning $\max_{j \in J}\{f(x_j) + \langle g_j, x - x_j\rangle\}$ with $g_j \in \partial f(x_j)$. At each iteration, given a previous iterate $x_i$ and a model $\model{i}$, the method computes a candidate iterate by taking a proximal step on the model,
\begin{equation}\label{eqn: Euclidean proximal model candidate}
    z_{i+1} \leftarrow \argmin_{z \in \reals^d} \left\{\model{i}(z) + \frac{\rho_i}{2}\norm{z - x_i}^2\right\} \qquad \text{for some }\rho_i >0.
\end{equation}
The next iterate $x_{i+1}$ is not necessarily updated to the candidate $z_{i+1}$. Instead, the method checks whether the decrease in function value at $z_{i+1}$ is at least a $\beta$ fraction of the decrease predicted by the model---intuitively, this means that the model is a sufficiently good approximation at $z_{i+1}$. If so, the method takes a \textit{Descent Step}, i.e., $x_{i+1} \leftarrow z_{i+1}$. Otherwise, it takes a \textit{Null Step}, i.e., $x_{i+1} \leftarrow x_i$. In both cases, the model $\model{i+1}$ is updated using a new subgradient $g_{i+1} \in \partial f(z_{i+1})$. For the reader's convenience and to draw a parallel with the Riemannian bundle method introduced in the next section, we include pseudocode in Algorithm~\ref{alg:euclidean-PBM}. 

\begin{algorithm}[t]
\caption{Euclidean Proximal Bundle Method}\label{alg:euclidean-PBM}
\begin{algorithmic}
\State \textbf{Input:} $x_0 \in \reals^d, g_0 \in \partial f(x_0),  \model{0}(z) = f(x_0)+ \ip{g_0, z - x_0}, \beta \in (0,1)$, $\rho_{0} > 0.$ \vspace{.2cm}
\For{$i = 0,1,2,\dots$}
    \State Compute candidate iterate $z_{i+1} \leftarrow \argmin_{z \in \reals^d}\{\model{i}(z) + \frac{\rho_i}{2} \norm{z-x_i}^2\}$.  \vspace{.1cm}
    \If{$\beta(f(x_i) - \model{i}(z_{i+1}))\leq f(x_i) - f(z_{i+1})$} 
    \State Set $x_{i+1}  \leftarrow z_{i+1}$, \Comment{Descent Step}
    \Else 
    \State Set $x_{i+1} \leftarrow x_i$. \Comment{Null Step}
    \EndIf \vspace{.1cm}
\State Update $\model{i+1}$ and $\rho_{i+1}$, without violating mild assumptions.
\EndFor
\end{algorithmic}
\end{algorithm}

Our analysis builds on \cite{díaz2023optimal}, which established optimal convergence rates for various function classes in the Euclidean setting. We recall the key properties their models must satisfy, as these will guide our Riemannian generalization. Specifically, the model $\model{i+1}$ must be a global minorant of the objective, $\model{i+1} \leq f$, and satisfy that for all $x \in \RR^d$ the following lower bounds hold
\begin{align}
    \model{i+1}(x) &\geq f(z_{i+1}) + \ip{g_{i+1}, x - z_{i+1}}, \label{assump:euclidean-subg-lb}\\
    \model{i+1}(x) &\geq \model{i}(z_{i+1})+ \ip{s_{i+1},x-z_{i+1}}, \label{assump:euclidean-mod-subg-lb}
\end{align}
where $g_{i+1} \in \partial f(z_{i+1})$ and $s_{i+1} = \rho_i(x_i - z_{i+1}) \in \partial \model{i}(z_{i+1})$ (the inclusion holds by first-order optimality of \eqref{eqn: Euclidean proximal model candidate}). The second lower bound is only required on null steps. The linearization cut \eqref{assump:euclidean-subg-lb} incorporates new first-order information at $z_{i+1}$, while the aggregate cut \eqref{assump:euclidean-mod-subg-lb} retains approximation accuracy from the previous model.

\section{Riemannian Convex Bundle Method}
In this section, we introduce our Riemannian proximal bundle method and establish convergence guarantees. Rather than presenting the algorithm in final form at the outset, we develop it systematically from first principles, with particular emphasis on the role of curvature in algorithmic design. Section~\ref{sec:curvature} introduces the assumptions on the manifold and basic primitives; Section~\ref{sec:local} develops local minorant models; Section~\ref{sec:algo} presents the resulting algorithm and explains how it handles model inexactness; Section~\ref{sec:model-ass} states the abstract model assumptions; and Section~\ref{sec:convergence-rates} provides the convergence analysis. 

Before continuing, we make the following blanket assumption.
\begin{assumption} \label{assump: bounded sublevel sets}
    For all $c \in \mathbb{R}$, there exists $z \in \mathcal{M}$ and $D_{z} \geq 0$ such that \begin{equation*}
        \left\{ x \in \cM \colon f(x) \leq c \right\} \subseteq \left\{x \in \cM \colon d_\mathcal{M}(x,z) \leq D_{z}\right\}.
    \end{equation*}
\end{assumption}
This assumption ensures that the iterates of our algorithm remain bounded,
as our method (\Cref{alg: Adaptive Riemannian Proximal Bundle Methods}) guarantees a monotonic decrease in objective value when updating iterates.
Existing Riemannian bundle methods \cite{bergmann2025convexbundle, hoseini2023proximal} assume bounded $\dom{f}$, which is strictly stronger. 
Euclidean bundle methods do not require this assumption, as it is well known that their iterates remain bounded. The necessity of this assumption for the Riemannian setting remains an open question.


\subsection{Assumptions on manifold primitives}\label{sec:curvature}
To generalize the Euclidean bundle method (Algorithm~\ref{alg:euclidean-PBM}) to the manifold setting, we must handle proximal subproblems of the form $\min_{z \in \cM}\{f(z) + \frac{\rho}{2}d_\cM^2(z,x)\}$. On Hadamard manifolds, the Cartan--Hadamard theorem guarantees that $\exp_x \colon T_x\cM \to \cM$ is a global diffeomorphism, so this subproblem can be equivalently reformulated in the tangent space via
\begin{equation} \label{eqn:prox-original-fun}
    \min_{z \in \cM} \left\{f(z) + \frac{\rho}{2}d_{\cM}^2(z,x)\right\} = \min_{v \in T_{x}\mathcal{M}} \left\{ f(\exp_{x}(v)) + \frac{\rho}{2} \norm{v}_x^{2} \right\}.
  \end{equation}
  The right-hand side is a convex optimization problem over the Euclidean space $T_x\cM$ equipped with $\langle \cdot, \cdot \rangle_x$.\footnote{Indeed, this is a standard proximal operator, up to a change of variables to transform $\|\cdot\|_x$ into the canonical Euclidean norm.} To mirror the global diffeomorphism property of $\exp_x$ when replacing it with a retraction $R_x$, we impose the following.
  \begin{assumption}[\textbf{Global primitives}]\label{assump:global-manifold-primitives} The retraction $R_x \colon T_x\cM \to \cM$ is a global diffeomorphism for any $x \in \cM$, and the transporter $\cT$ is defined on 
  $\cM \times \cM$.
    \end{assumption}

This assumption allows us consider a problem akin to \eqref{eqn:prox-original-fun} with $R_x$ in place of $\exp_x.$ Additionally, we impose concrete bounds on the inexactness of the retraction and transporter.

\begin{assumption}[\textbf{Inexact primitives}]\label{assump:inexact-primitive-control} The following two hold. 
    \begin{enumerate}[leftmargin=3em]
        \item (\textbf{Local retraction error}) For any given compact set $K \subseteq \cM$ and  $a > 0$ there exists a $\CRet>0$ such that for any $x \in K$ and $v \in T_{x}\cM$ with $\|v \|_{x} \leq a$, the retraction $R_x$ satisfies
        \begin{equation}\label{eqn: retraction error}
            \manDist{\exp_x(v)}{R_x(v)} \leq \CRet \norm{v}_x^2.
        \end{equation}

      \item (\textbf{Transporter error}) There exists a constant $\CVec > 0$ such that for any $x, y \in \cM$, the transporter $\VecT{x}{y}$ 
            satisfies
        \begin{equation}\label{eqn: transport error}
            \norm{\VecT{x}{y}(v) - \ParT{x}{y}(v)}_x \leq \CVec \norm{v}_y \manDist{x}{y} \qquad \text{for all } v \in T_y\cM.
        \end{equation}
    \end{enumerate}
\end{assumption}

A few comments are in order. Intuitively, retractions agree with the exponential map to first order at each point, so their discrepancy is locally quadratic in the size of the vector mapped, as originally argued in \cite{Absil2012}. Additionally, the smoothness of retractions by Assumption~\ref{assump:global-manifold-primitives} ensures that a uniform constant can hold over all of $K$. We will set a concrete set $K$ and a constant $a$ for our algorithms analysis a posteriori.
The transporter error bound \eqref{eqn: transport error}, originally introduced in \cite{Li-zero-order-riem-2024}, dictates that the deviation from parallel transport grows at most linearly with the manifold distance. 
A prototypical map satisfying \eqref{eqn: transport error} is the projection transport, provided that $\cM \subseteq \reals^d$ is an embedded submanifold with a bounded second fundamental form \cite{Li-zero-order-riem-2024}.
Formally, the projection transport is defined as
\begin{equation}\label{eqn: projection transport}
    v \mapsto \operatorname{Proj}_{T_x\cM}(v) \qquad \text{for all }v \in T_y \cM,
\end{equation}
where $\operatorname{Proj}_{T_x\cM}$ is the orthogonal projection from the ambient space $\reals^d$ onto $T_x\cM$.


Finally, we impose a standard bounded curvature assumption.
\begin{assumption}[\textbf{Bounded sectional curvature}] \label{assump:bounded-sectional-curvature} There exists $K_{\operatorname{min}} > -\infty$, such that $K_{\operatorname{min}}\leq K_{x}(v_{x},w_{x})$ for all $x \in \cM$ and linearly independent $v_{x},w_{x} \in T_{x}\cM$.
\end{assumption}


\subsection{Local models and curvature}\label{sec:local} 
As in the Euclidean case, we replace the proximal step on the original function~\eqref{eqn:prox-original-fun} with one on a local model, which now will be defined on the tangent space. Formally, at iteration $i \geq 0$, the method uses a convex model $\model{i} \colon T_{x_i}\mathcal{M} \rightarrow \mathbb{R}$ to compute a candidate direction,
\begin{equation} \label{eqn: candidate direction subproblem}
\canDir{i+1} := \operatorname*{argmin}_{v \in T_{x_{i}}\mathcal{M}} \left\{ \model{i}(v) + \frac{\rho_{i}}{2}\lVert v \rVert ^{2} \right\}.
\end{equation}
For structured models, such as a pointwise maximum of affine functions, this subproblem admits closed-form or QP-based solutions. The candidate iterate is then obtained via retraction, namely $z_{i+1} := R_{x_{i}}(\canDir{i+1}) \in \mathcal{M}$. In the Euclidean setting, it is standard to build globally minorant models from subgradients, i.e., $\model{i} \leq f$ on all of $\RR^d.$ Lower-bounding models are crucial for obtaining convergent algorithms. On curved spaces, however, transporting subgradients across tangent spaces introduces errors that prevent the construction of models that are global minorants. 

 To illustrate this point, recall that 
$g$-convexity of $f$ gives the global subgradient inequality $f(y) \geq f(z) + \langle g, \log_z(y)\rangle_z$ for all $y \in \cM$ when $g \in \partial f(z)$. This inequality is `centered' at $z.$ As our algorithm moves to new iterates, we might want to recenter this inequality to some other point $x$, to do so we apply the parallel transport map to both inputs of the inner product. If we focus on the Euclidean setting for just a moment, we have $\ParT{x}{z}[\log_z(y)] = \log_x(y) - \log_x(z)$, so this yields
$
    f(y) \geq f(z) + \langle \ParT{x}{z}[g], \log_x(y) - \log_x(z) \rangle_x.$ 
If we want a model on the tangent space at $x,$ we can think of $v = \log_x(y) = y - x$ as a vector in $T_x \cM$. So, the model $v \mapsto f(z) + \langle \ParT{x}{z}[g], v - \log_x(z) \rangle_x$ is a minorant in the sense that
\begin{equation}
f(\exp_x(v)) \geq f(z) + \langle \ParT{x}{z}[g], v - \log_x(z) \rangle_x.
\end{equation}

Naturally, one might expect to construct similar models on Hadamard manifolds. However, this inequality fails to hold in our setting because of three sources of error: $(i)$ the parallel transport identity $\ParT{x}{z}[\log_z(y)] = \log_x(y) - \log_x(z)$ breaks under nonzero sectional curvature; $(ii)$ approximating the exponential map by a first-order retraction; and $(iii)$ approximating parallel transport by a transporter. 
Nevertheless, as we show next, we may ask for the `transported' inequality to hold only \textit{locally}, when paired with an appropriate affine shift.

Let $\alpha > 0$ be given. If $g \in \partial f(z)$ for $z \in B_{\cM}(x,\alpha)$, define an affine shift by
\begin{equation}\label{eqn:general-error-shift}
    \genShift{\alpha}{g} := (2 \sqrt{-K_{\min}} + \CRet + 2\CVec) \norm{g}\alpha^{2}.
\end{equation}
This shift captures errors introduced by curvature and inexact primitives. Indeed, under Assumption~\ref{assump:inexact-primitive-control}, the inexact primitive error scales linearly with the transport distance. 
Under bounded sectional curvature, deviations in the parallel transport identity scale similarly. The following formalizes that this shift is sufficient; we defer its proof to \Cref{appendix: proof of error adjustment}.

\begin{lemma} \label{lem: error adjustment shift}
    Suppose that Assumptions~\ref{assump:global-manifold-primitives},~\ref{assump:inexact-primitive-control}, and~\ref{assump:bounded-sectional-curvature} hold. Let compact $K\subseteq \mathcal{M}$ and $a > 0$ be given. Let $x \in K$, $\alpha$ be such $0 < \alpha \leq a$, and $\CRet$ be the corresponding retraction constant from \eqref{eqn: retraction error}. Let $z = R_{x}(v_{z}) \in B_{\cM}(x,\alpha)$ with $g \in \partial f(z)$. Then, for all $v \in B_{x}(0,\alpha)$,
    \begin{equation*}
        f(\exp_{x}(v)) \geq f(z) + \ip{\VecT{x}{z}[g], v - v_{z}}_{x} - \genShift{\alpha}{g}.
    \end{equation*}
\end{lemma}

To apply this result algorithmically, we associate each model $\model{i}$ and proximal parameter $\rho_i$ with a transport radius $\alpha=\alpha(\model{i},\rho_i)$, defined as the maximum distance the corresponding proximal iterate is allowed to move. As in the Euclidean case, larger values of $\rho_i$ penalize longer steps more strongly and therefore decrease the admissible radius $\alpha$.
To keep the framework flexible with respect to the choice of model, we avoid imposing strong structural assumptions on $\model{i}$. Instead, we assume they are lower-bounded by models defined via subgradients. 
The next lemma formalizes these ideas; its proof follows from a simple level-set argument and is therefore omitted.


\begin{lemma}\label{lem: candidate direction radius}
    Fix $x \in \mathcal{M}$, $g \in T_{x}\mathcal{M}$, and $\rho > 0$. Let $\tilde{f} \colon T_{x}\mathcal{M} \rightarrow \mathbb{R}$ satisfy $\tilde{f}(0) = f(x)$ and $f(x) + \ip{g, v}_{x} \leq \tilde{f}(v)$ for all $v \in T_{x}\mathcal{M}$. If $v_{\ast} = \argmin \{\tilde{f}(v) + \frac{\rho}{2}\norm{v}^2 \}$, then $\norm{v_{\ast}}_{x} \leq 2\norm{g}_{x}/\rho.$
\end{lemma}


From now on we set $K := \{x \colon f(x) \leq f(x_0)\}$ and $a = \fLip/\rho_0$.
By Assumption~\ref{assump: bounded sublevel sets} and the Hopf--Rinow theorem, the initial sublevel set $K$ is compact. Since $\rho_i \geq \rho_0$, Lemma~\ref{lem: candidate direction radius} together with the retraction error bound implies that any solution $\canDir{i+1}$ of \eqref{eqn: candidate direction subproblem} yields a candidate $z_{i+1}=R_{x_i}(\canDir{i+1})$ with $z_{i+1}\in B_{\cM}(x_i,\candR{i})$, where
\begin{equation}\label{eqn: radius of candidates}
    \candR{i} := \frac{2\norm{g}_{x_i}}{\rho_i}
    + \CRet\left(\frac{2\norm{g}_{x_i}}{\rho_i}\right)^2 .
\end{equation}
We therefore use $\alpha=\candR{i}$ in \eqref{eqn:general-error-shift} and define the corresponding intercept shift by
\begin{equation}\label{eqn: intercept shift}
    \modShift{i+1}
    := \kappa(\candR{i},g_{i+1})
    = \left(2\sqrt{-K_{\operatorname{min}}}+\CRet+2\CVec\right)
    \norm{g_{i+1}}(\candR{i})^2 .
\end{equation}
With this choice, $\model{i}(v)\leq f(\exp_{x_i}(v))$ for all $v\in B_{x_i}(0_{x_i},2\norm{g}_{x_i}/\rho_i)$. Henceforth, $K$ and $a$ refer to the above quantities instantiating Assumption~\ref{assump:inexact-primitive-control}.


\subsection{Algorithmic Statement} \label{sec:algo}
Next, we introduce our generalization of the proximal bundle method to Hadamard manifolds, summarized in Algorithm~\ref{alg: Adaptive Riemannian Proximal Bundle Methods}. The main distinction from the Euclidean setting is the role of the proximal parameter \(\rho_i\). As shown already, local models are reliable only near previous iterates, and thus \(\rho_i\) must be sufficiently large to prevent the iterates from straying too far. In what follows, we show how to account for this curvature. 

A key object in establishing guarantees for proximal bundle methods is the model proximal subproblem gap, namely 
\begin{equation*}
    \modProxGap{i} := f(x_{i}) - \min_{v} \left\{\model{i}(v)+ \frac{\rho_{i}}{2} \norm{v}^{2}\right\}.
\end{equation*}
This gap governs both the decrease obtained at descent steps and the number of null steps between consecutive descent steps; see, for instance, \cite{díaz2023optimal} or Section~\ref{subsec: key lemmas}. Accordingly, our method requires the model shift \(\modShift{i+1}\) to be smaller than \(\modProxGap{i}\), since otherwise the model inexactness could outweigh the predicted progress.

Consequently, we wish to choose the proximal parameter $\rho_i$ to either ensure descent or control the model shift. Formally, bundle methods determines if a candidate iterate $z_{i+1}$ achieves a sufficient reduction in the objective value via
\begin{equation}
    f(x_{i}) - f(z_{i+1}) \geq \beta(f(x_{i}) - \model{i}(\canDir{i+1})), \label{eqn: descent criterion}
\end{equation}
we do not modify this criterion. To ensure that the model shift is sufficiently controlled, we require 
\begin{equation}
    \frac{1}{2} \modProxGap{i} - \frac{\modShift{i+1}}{1-\beta} \geq 0, \label{eqn: sufficient null progress}
\end{equation}
Our method (summarized in Algorithm~\ref{alg: Adaptive Riemannian Proximal Bundle Methods}) chooses $\rho_i$ so that either \eqref{eqn: descent criterion} or \eqref{eqn: sufficient null progress} holds, 
which may be found algorithmically through backtracking. For instance, by doubling via \Cref{alg:linesearch}. The following next lemma, a special case of \Cref{lem: stepsize controls descent for recurrence}, shows that the the backtracking method finishes after finitely many steps. 

\begin{lemma} \label{lem: uniform lower bound for proximal} 
Fix $\varepsilon > 0$. Let $A := \frac{16 (2\sqrt{-K_{\operatorname{min}}} + 2\CVec + \CRet) (1+\CRet )^{2} \fLip^{3}}{1-\beta}$ and 
\begin{equation} \label{eqn: constant proximal parameter lower bound inexact}
    \tilde{\rho} = \max \left\{\left(\frac{A^{1/2}D}{\varepsilon}\right)^{2}, \left(\frac{A^{1/2}\fLip^{1/2}D}{\varepsilon}\right)^{2/3}, \left(\frac{A}{\varepsilon}\right)^{1/2}, \left(\frac{A\fLip}{\varepsilon}\right)^{1/4} \right\},
\end{equation}
where $D = \sup_{i}\operatorname{dist}(x_{i}, \minimizers)$.
Suppose that $\rho_i$ and $x_i$ in Algorithm~\ref{alg: Adaptive Riemannian Proximal Bundle Methods} satisfy $\rho_{i} \geq \tilde{\rho}$ and $f(x_{i}) - f_{\ast} \geq \varepsilon$. Then, either \eqref{eqn: descent criterion} or \eqref{eqn: sufficient null progress} holds. Hence, Algorithm~\ref{alg:linesearch} terminates within $\lceil \log_{2}(\tilde{\rho}/\rho_{0})\rceil_{+}$ iterations during a run of Algorithm~\ref{alg: Adaptive Riemannian Proximal Bundle Methods}. 
\end{lemma}

When $\cM = \mathbb{R}^{d}$, $K_{\operatorname{min}} = \CRet = \CVec =0$, $\tilde{\rho} = 0$ and no backtracking occurs.

\begin{algorithm}[t]
\caption{Riemannian proximal bundle method}\label{alg: Adaptive Riemannian Proximal Bundle Methods}
\begin{algorithmic}
\State \textbf{Input:} $z_0 = x_0 \in \cM,\; \model{0} = f(x_0)+ \ip{g_0, v},\; \rho_0 > 0$. \vspace{.2cm}
\For{$i = 0,1,2,\dots$}

\State Select $\rho_i$, such that \Comment{For instance with Alg.~\ref{alg:linesearch}}
\[
\canDir{i+1}
\leftarrow
\argmin_{v \in T_{x_i}\cM}
\Bigl\{\model{i}(v) + \frac{\rho_i}{2}\norm{v}^2\Bigr\}
\quad \text{and}\quad
z_{i+1} \leftarrow R_{x_i}\left(\canDir{i+1}\right),
\]
    \State and $g_{i+1} \in \partial f(z_{i+1})$ satisfy 
    either~\eqref{eqn: descent criterion} or \eqref{eqn: sufficient null progress}. \vspace{.1cm}
    
    \If{$\beta(f(x_{i}) - \model{i}(\canDir{i+1})) \leq f(x_{i}) - f(z_{i+1})$} 
        \State Set $x_{i+1} \leftarrow z_{i+1},$ \Comment{Descent Step}
        
    \Else
        \State Set $x_{i+1} \leftarrow x_i$. \Comment{Null Step}
    \EndIf \vspace{.1cm}
    
    \State Update $\hat{f}_{i+1}$ such that Assumption~\ref{assump: imperfect setting} holds.
    
\EndFor
\end{algorithmic}
\end{algorithm}

\begin{algorithm}[t]
\caption{Proximal parameter backtracking}\label{alg:linesearch}
\begin{algorithmic}
\State \textbf{Input:} Constants $\beta, \fLip, \CRet, \CVec,$ and $K_{\operatorname{min}},$ center $x$, model $\widehat f \colon T_{x}\mathcal{M} \rightarrow \mathbb{R}$, retraction $R_x$, and initial $\rho$. \vspace{.1cm}
\State Compute $\modShift{}$ as in \eqref{eqn: intercept shift}, $d$ as in \eqref{eqn: candidate direction subproblem}, and set $z = R_{x}(v)$.\vspace{.1cm}
\While{neither~\eqref{eqn: descent criterion} nor \eqref{eqn: sufficient null progress} are satisfied}
\State Double $\rho \leftarrow 2 \rho.$
\State Recompute $\modShift{}$, $d$, and $z$.
\EndWhile 
\State \Return $\rho$.
\end{algorithmic}
\end{algorithm}


\subsection{Model and Function Assumptions}\label{sec:model-ass}

By the first-order optimality conditions, 
the \textit{model subgradient} is
$s_{i+1} := -\rho_{i} \canDir{i+1} \in \partial \model{i}(\canDir{i+1}) \subseteq T_{x_{i}}\cM$ for all $i \geq 0$.
We assume oracle access to $x \mapsto (f(x), g(x))$, $g(x) \in \partial f(x)$, for any $x \in \cM$.

\begin{assumption}\label{assump: imperfect setting} Let $\{\model{i}: T_{x_{i} }\cM \rightarrow \reals \}_{i \geq 0}$, $\{x_{i} \}_{i \geq 0}$, and $\{\rho_{i}\}_{i\geq 0}$ be local convex models, proximal centers, and proximal parameters used in a run of \Cref{alg: Adaptive Riemannian Proximal Bundle Methods}. 
Let $k \leq i$ denote the index of the most recent descent step (i.e., $x_{i+1} = x_{k+1}$) with $g_{k+1} \in \partial f(x_{k+1})$. Let $g_{i+1} \in \partial f(z_{i+1})$. Let $\rho_{i+1} \geq \rho_{i}$ and $\model{i+1}$ satisfy the following.
\begin{enumerate}[leftmargin=3em]
    \item \textbf{Local Minorant of Objective.} If $v \in B_{x_{i+1}}(0,2\norm{g_{k+1}}/\rho_{i+1}) \subseteq T_{x_{i+1}}\cM$,\begin{equation}  \label{eqn: minorant, ret/vec}
        \model{i+1}(v) \leq f (\exp_{x_{i+1} }(v)). 
    \end{equation}
    \item \textbf{Null-Step Model Refinements.} For null steps, $\model{i+1}$ additionally satisfies
    \begin{align}
        \model{i+1}(v) & \geq f(z_{i+1}) + \ip{\VecT{x_{i+1}}{z_{i+1}}[g_{i+1}], v-\canDir{i+1}}_{x_{i+1}} - \modShift{i+1} =: \ell_{\operatorname{new}}(v) \label{eqn: subg lower bound, ret/vec}\\
        \model{i+1}(v) & \geq \model{i}(\canDir{i+1}) + \ip{s_{i+1}, v- \canDir{i+1}}_{x_{i+1}} =: \ell_{\operatorname{agg}}(v)\label{eqn: model subg lower bound, ret/vec}
    \end{align} 
    for all $v \in B_{x_{i+1}}(0_{x_{i+1}}, 2 \norm{g_{k+1}}/\rho_{i+1}) \subseteq T_{x_{i+1}}\mathcal{M}$.

    \item \textbf{Anchor Lower Bound.} For all $v \in T_{x_{i+1}}\cM$, 
    \begin{equation} \label{eqn: anchor lower bound}
        \model{i+1}(v) \geq f(x_{k+1}) + \ip{g_{k+1},v}_{x_{i+1}} =: \ell_{\operatorname{anchor}}(v).
    \end{equation}

\end{enumerate}
\end{assumption}

The first two assumptions \textit{locally} mirror those in the Euclidean setting, with a suitable affine shift applied to maintain the local minorant property (see \Cref{lem: error adjustment shift}). To specify this neighborhood, the first and last ensure the candidate direction norms are explicitly bounded (see \Cref{lem: candidate direction radius}). Lastly, while \eqref{eqn: minorant, ret/vec} is stated in terms of the exponential map, we emphasize that $\exp_{x_{i+1}}$ is never evaluated 
by \Cref{alg: Adaptive Riemannian Proximal Bundle Methods}. 

\subsection{Guarantees} \label{sec:convergence-rates}
We provide guarantees in the $\fLip$-Lipschitz setting, with improved rates under function growth.

\textbf{Convergence Under Backtracking.} 
First, we define the following quantity
\begin{equation} \label{eqn: uniform upper bound on transported subgradient}
    \constModelLip{\rho_{0}} := \left(1 + \CVec \left[\frac{2\fLip}{\rho_{0}} + \CRet \left(\frac{2\fLip}{\rho_{0}}\right)^{2}\right] \right).
\end{equation}
In turn, the norms of transported subgradients during a run of Algorithm~\ref{alg: Adaptive Riemannian Proximal Bundle Methods} are uniformly bounded by $\constModelLip{\rho_{0}} \fLip$; see \Cref{claim: Upper bound on model Lip}. 
When using parallel transport, $\CVec = 0$ and thus $\constModelLip{\rho_{0}} = 1$.
\begin{theorem}\label{thm: convergence Lipschitz no growth} Suppose Assumptions~\ref{assump: bounded sublevel sets},~\ref{assump:global-manifold-primitives},~\ref{assump:inexact-primitive-control},~\ref{assump:bounded-sectional-curvature}, and~\ref{assump: imperfect setting} hold. Let $f$ be a $\fLip$-Lipschitz $g$-convex function, $\rho_{0} > 0$ be fixed, and $\tilde{\rho}$ be defined by \eqref{eqn: constant proximal parameter lower bound inexact}. Then,  to achieve an $\varepsilon$-minimizer (for any $\varepsilon > 0$) with \Cref{alg: Adaptive Riemannian Proximal Bundle Methods} with proximal parameter schedule produced by \Cref{alg:linesearch}, requires at most
\begin{equation*} 
\left\lceil  \frac{2 }{\beta}\log\left( \frac{f(x_{0}) - f_{\ast}}{\rho_{0}D^{2}} \right)\right\rceil _{+} + \frac{2 D^{2}}{\beta \varepsilon} \cdot \max\{\rho_{0}, 2 \tilde{\rho}\} \qquad \text{descent steps},
\end{equation*} 
and at most
\begin{equation*} 
\frac{256 \constModelLip{\rho_{0}}^{2} G_{f}^{2}}{3\beta(1-\beta)^{2}\rho_{0}^{2}D^{2}} + \frac{192 D^{4} \constModelLip{\rho_{0}}^{2} \fLip^{2}}{\beta(1-\beta)^{2} \varepsilon^{3}} \cdot \max\{\rho_{0}, 2 \tilde{\rho}\} \qquad \text{null steps}, 
\end{equation*}
where $D := \sup_{i}\operatorname{dist}(x_{i}, \minimizers)$. The number of backtracking steps is at most
\begin{equation*}
    \left \lceil \log_{2}\left(\frac{\tilde{\rho}}{\rho_{0}}\right)\right \rceil_{+}.
\end{equation*}
\end{theorem}

The resulting oracle complexity is $\mathcal{O}(\tilde{\rho} \cdot \varepsilon^{-3})$ recovering the classical complexity of \cite{kiwiel2000proximal}. By \Cref{lem: uniform lower bound for proximal}, $\tilde{\rho} \propto \mathcal{O}(\varepsilon^{-2})$ in the presence of curvature, implying an overall complexity of $\mathcal{O}(\epsilon^{-5})$. However, we note that $\tilde{\rho}$ continuously depends on $K_{\min}, \CVec, \CRet$, and this rate can be improved given a \textit{small} amount of curvature and error in our primitives. Indeed, in the Euclidean setting setting where $\tilde{\rho} = 0$, setting $\rho_{0} \propto \mathcal{O}(\epsilon)$ recovers the optimal oracle complexity of $\mathcal{O}(\epsilon^{-2})$ of \cite{díaz2023optimal}. 

\textbf{Improved Rates with Growth.} It is natural to wonder whether our bundle method can speed up in the presence of growth. Next, we show that this is indeed the case if we use an idealized proximal parameter. 
\begin{assumption}[\textbf{H\"{o}lder Growth}]\label{assump:p-Holder-growth} Let $p \geq 1$. There exists $\mu > 0$ such that 
\begin{equation}
    f(x) - f_{\ast} \geq \mu \cdot \operatorname{dist}(x, \minimizers)^{p} \quad \text{for all $x \in \cM$}.
\end{equation}
\end{assumption}
When, $p = 1$ or $p = 2$ this is known as $\mu$-sharpness or $\mu$-quadratic growth, respectively. To prove our guarantees, we consider the idealized (and often impractical) schedule
\begin{align} \label{eqn: specialized proximal parameter schedule}
    \rho_{k} = \max\big\{&A\mu^{-2/p}\delta_{k}^{(2-2p)/p}, (A\fLip\mu^{-2/p})^{1/3} \delta_{k}^{(2-2p)/3p}, (A/\delta_{k})^{1/2}, \\&(A\fLip/\delta_{k})^{1/4}, \mu^{2/p}\delta_{k}^{(p-2)/p}\big\} \nonumber
\end{align}
where $A = \frac{16 (2\sqrt{-K_{\operatorname{min}}} + 2\CVec + \CRet) (1+\CRet )^{2} \fLip^{3}}{1-\beta}$ and $\delta_{k} := f(x_{k}) - f_{\ast}$. 
Thus, implementing this schedule requires prior knowledge of the optimal objective value, which is typically unavailable in practice. The following result is therefore best understood as a proof of concept: it establishes acceleration under an idealized, impractical schedule. We expect similar guarantees may be obtainable via more sophisticated parallel methods, akin to \cite{díaz2023optimal}, or via an adaptive schedule; we leave these directions for future work.
\begin{theorem} \label{thm: convergence rate with growth tuned}
    Suppose Assumptions~\ref{assump: bounded sublevel sets},~\ref{assump:global-manifold-primitives},~\ref{assump:inexact-primitive-control},~\ref{assump:bounded-sectional-curvature}, and~\ref{assump: imperfect setting} hold. Let $f$ be a $\fLip$-Lipschitz  $g$-convex function satisfying \Cref{assump:p-Holder-growth}. Then, running \Cref{alg: Adaptive Riemannian Proximal Bundle Methods} with proximal parameter schedule \eqref{eqn: specialized proximal parameter schedule}, the number of descent steps required to achieve an $\varepsilon$-minimizer (for any $\varepsilon > 0$) is at most
    $$
    \begin{cases}
        \mathcal{O}(\log(1/\varepsilon)), \quad & p \in [1,4/3],\\
        \mathcal{O}(1/\varepsilon^{3-4/p}), & p > 4/3,\\
    \end{cases}
    $$
    and the number of null steps is at most
    $$ \begin{cases}
        \mathcal{O}(\log(1/\varepsilon)), \quad & p=1,\\
        \mathcal{O}(1/\varepsilon^{2-2/p}), & p\in (1,4/3],\\
        \mathcal{O}(1/\varepsilon^{5-6/p}), & p > 4/3.\\
    \end{cases}$$

\end{theorem}

When $p \in [1,4/3]$, the method obtains minimax optimal oracle complexity. When $p > 4/3$, oracle complexity improves relative to \Cref{thm: convergence Lipschitz no growth}, degrading to $\mathcal{O}(\varepsilon^{-5})$ as $p \rightarrow \infty$. We note that in the Euclidean case, the rates for $p > 4/3$ improve from $\mathcal{O}(1/\varepsilon^{5-6/p})$ to the optimal $\mathcal{O}(1/\varepsilon^{2-2/p})$, as the schedule reduces to that of Theorem 2.6 \cite{díaz2023optimal} enabling the argument of $p \in (1, 4/3]$ to extend to $p > 4/3$. The explicit bounds that recover the stated complexities are given in \Cref{subsec:proof-of-improved-convergence}.

\section{Proofs}

Let $i$ denote a general iteration, $k$ the last descent step iteration, and $T$ the number of null steps at the current proximal center. 

\subsection{Key Lemmas and Proof Strategy}\label{subsec: key lemmas}
Define the \textit{proximal gap} by
\begin{equation*}
    \Delta_i := f(x_i) - \left(f\left(\exp_{x_i}(\trueProxStep)\right) + \frac{\rho_i}{2}\big \lVert\trueProxStep \big \rVert^2\right)
\end{equation*}
where $\trueProxStep = \operatorname*{argmin}_{v \in T_{x_{i}}\cM} \{f\left(\exp_{x_i}(v)\right) + \frac{\rho_i}{2}\norm{v}^2 \}$. The proximal gap $\Delta_i$ directly controls the behavior of descent steps and null steps.


\begin{enumerate}[label=(\roman*), leftmargin=2.5em]
    \item \textbf{Descent steps attain a decrease proportional to the proximal gap.}
    
    \begin{lemma}[Descent Guarantee] \label{lem: descent guarantee} Suppose \Cref{assump: imperfect setting} holds. For any descent step $k$, $f(x_{k+1}) \leq f(x_{k}) - \beta \Delta_k.$
    \end{lemma}

    \item \textbf{Proximal gap bounds the number of null steps at a center.}

    \begin{lemma}[Null Step Bound] \label{lem: null step bound imperfect info}
    Suppose Assumptions of \Cref{thm: convergence Lipschitz no growth} hold. Consider a descent step at iteration $k$ followed by $T$ null steps. Then, 
    \begin{equation*}
        T_{\operatorname{null}} \leq \frac{32\constModelLip{\rho_{0}}^{2}\fLip^{2}}{(1-\beta)^2 \rho_{k+1} \Delta_{k+T}}.
    \end{equation*}
    \end{lemma}
\end{enumerate}
Proofs are given in \Cref{subsec: proof of descent step guarantee} and \Cref{subsec: proof of null step gurantee}. Focusing on proximal gaps for fixed $\rho_i$, we extend \cite{ruszczynski2006nonlinear} to our setting. The proof is deferred to \Cref{subsec: proof of connect prox step to objective gap}.    

\begin{lemma} \label{lem: connect prox step to objective gap} Suppose $f \colon \cM \rightarrow \mathbb{R}$ is $g$-convex. Let $\mathcal{X}_{\ast} = \argmin f$, and $x_{i} \in \cM \setminus \mathcal{X}_{\ast}$. Then, the proximal gap satisfies
    \begin{equation}\label{eqn: prox gap lower bound by obj gap and dist}
        \Delta_i \geq \begin{cases}
        \displaystyle
            \frac{1}{2\rho_i}\left(\frac{f(x_{i}) - f_{\ast}}{d_{\mathcal{M}}(x_{i}, x_{\ast})}\right)^2, \quad & \text{if } f(x_{i}) - f_{\ast} \leq \rho_i d_{\mathcal{M}}^{2}(x_{i}, x_{\ast}),\\
            f(x_{i}) - f_{\ast} - \frac{\rho_i}{2}d_{\mathcal{M}}^{2}(x_{i}, x_{\ast}), & \text{otherwise},
        \end{cases}
    \end{equation}
    where $x_{\ast} \in \argmin_{y \in \mathcal{X}_{\ast}} d_\mathcal{M}(x_{i},y)$.
\end{lemma}

Finally, we provide a lower bound so proximal parameters remain constant at proximal centers when running \Cref{alg: Adaptive Riemannian Proximal Bundle Methods}.
The proof is given in \Cref{subsec: proof of uniform lower bound for proximal parameter}.


\begin{lemma}\label{lem: stepsize controls descent for recurrence}
     Let $k$ be a descent step followed by $T$ consecutive null steps at proximal center $x_{k+1} \in \cM$. Let $f$ is a $\fLip$-Lipschitz $g$-convex function, $\delta_{k+1} = f(x_{k+1}) - f_{\ast}$, and $D_{k+1} = \operatorname{dist}(x_{k+1}, \minimizers)$. If
     \begin{equation} \label{eqn: form for lower bound on proximal parameters}
        \rho \geq \max \left\{\left(\frac{A^{1/2}D_{k+1}}{\delta_{k+1}}\right)^{2}, \left(\frac{A^{1/2}\fLip^{1/2}D_{k+1}}{\delta_{k+1}}\right)^{2/3}, \left(\frac{A}{\delta_{k+1}}\right)^{1/2}, \left(\frac{A\fLip}{\delta_{k+1}}\right)^{1/4}\right\},
     \end{equation}
     where $A := \frac{16 (2\sqrt{-K_{\operatorname{min}}} + 2\CVec + \CRet) (1+\CRet )^{2} \fLip^{3}}{1-\beta}$, then \eqref{eqn: sufficient null progress} holds for $k+1 \leq t \leq k+T$.
    If Assumption~\ref{assump:p-Holder-growth} also holds, \eqref{eqn: form for lower bound on proximal parameters} simplifies to
    \begin{equation} \label{eqn: form for lower bound on proximal parameters, Holder}
        \rho \geq \max \left\{A\mu^{-2/p}\delta_{k+1}^{(2-2p)/p}, (A\fLip\mu^{-2/p})^{1/3} \delta_{k+1}^{(2-2p)/3p}, (A/\delta_{k+1})^{1/2}, (A\fLip/\delta_{k})^{1/4}\right\}.
     \end{equation}
\end{lemma}
We obtain the bound of \Cref{lem: uniform lower bound for proximal} by replacing $D_{k+1}$ with $D$ and $\delta_{k+1}$ with $\epsilon$.

To derive convergence rates, we bound the total number of descent and null steps. Combining Lemma~\ref{lem: descent guarantee} with Lemma~\ref{lem: connect prox step to objective gap} yields a recurrence relation for the descent step bound, while the combining Lemma~\ref{lem: null step bound imperfect info} and Lemma~\ref{lem: connect prox step to objective gap} gives a null-step bound. Lastly, Lemma~\ref{lem: uniform lower bound for proximal} allows us to handle varying parameters through an algorithm's run due to backtracking.


Before proceeding we highlight the following auxiliary result, allowing us to upper-bound the norm of transported subgradients. We omit details, as they follow from the primitive assumptions.

\begin{lemma} \label{claim: Upper bound on model Lip} Suppose Assumptions~\ref{assump:inexact-primitive-control} and~\ref{assump: imperfect setting} hold. Given $g_{i+1} \in \partial f(z_{i+1}),$ with $z_{i+1} = R_{x_{i}}(d_{i+1})$ and $d_{i+1}$ defined by \eqref{eqn: candidate direction subproblem}, 
\begin{equation} \label{eqn: local bound on transported subg norm}
    \norm{\VecT{x_{i}}{z_{i+1}}[g_{i+1}]}_{x_{i}} \leq \left(1 + \CVec\left[\frac{2\fLip}{\rho_{i}} + \CRet \left(\frac{2 \fLip}{\rho_{i}}\right)^{2} \right]\right) \fLip.
\end{equation}
\end{lemma}

\subsubsection{Proof of \texorpdfstring{\Cref{lem: descent guarantee}}{Lemma 4}}\label{subsec: proof of descent step guarantee}
Let 
$\canDir{k+1} = \operatorname*{argmin}_{v \in T_{x_k} \cM}\{\model{k}(v) + \frac{\rho_k}{2}\norm{v}^2\}$
and $\trueProxStep  = \operatorname*{argmin}_{v \in T_{x_k} \cM} \{f(\exp_{x_k}(v)) + \frac{\rho_k}{2}\norm{v}^2\}.$ Then,
\begin{align*}
    \model{k}(\canDir{k+1}) \leq \model{k}(\canDir{k+1}) + \frac{\rho_k}{2}\lVert \canDir{k+1} \rVert^2
    \leq \model{k}(\trueProxStep) + \frac{\rho_k}{2} \lVert\trueProxStep  \rVert^2
    \leq f(\exp_{x_k}(\trueProxStep )) + \frac{\rho_k}{2} \lVert\trueProxStep  \rVert^2.
\end{align*}
The second inequality follows from the definition of $\canDir{k+1}$, and the last from applying \eqref{eqn: minorant, ret/vec} as $\bar{v}$ lies in the minorant neighborhood by
\Cref{lem: candidate direction radius}. Thus, $f(x_k) - \model{k}(\canDir{k+1}) \geq \Delta_k$. Lastly, $(f(x_k) - f(x_{k+1}))/\beta \geq \Delta_k$ follows from the definition of a descent step.

\subsubsection{Proof of \texorpdfstring{\Cref{lem: null step bound imperfect info}}{Lemma 5}} \label{subsec: proof of null step gurantee} 
The following recurrence is key for this lemma.
\begin{lemma}[Recurrence on Model Proximal Gap]\label{lem: master reccurence for null bound}
Suppose assumptions of \Cref{lem: null step bound imperfect info} hold. Let $x_{k+1}  \in \cM$ be the current proximal center. Given null step $k+1 \leq t \leq k+T$, where $T$ is the final consecutive null step at $x_{k+1} $, we have
\begin{equation}\label{eqn: master recurrence}
    \modProxGap{t+1} \leq \modProxGap{t} - \frac{(1-\beta)^2 \rho_{k+1}}{8\constModelLip{\rho_{0}}^{2}\fLip^2} \left(\modProxGap{t} - \frac{\modShift{t+1}}{(1-\beta)}\right)^2.
\end{equation}
\end{lemma}

 We show how this leads to the desired null step bound \Cref{lem: null step bound imperfect info}.
We have 
\begin{equation} \label{eqn: simplified master recurrence}
    \modProxGap{t+1} \leq \modProxGap{t} - \frac{(1-\beta)^2 \rho_{k+1}}{8\constModelLip{\rho_{0}}^{2}\fLip^{2}} \left(\modProxGap{t} - \frac{\modShift{t+1}}{(1-\beta)}\right)^2 \leq  \modProxGap{t} - \frac{(1-\beta)^2 \rho_{k+1} (\modProxGap{t})^2}{32 \constModelLip{\rho_{0}}^{2}\fLip^2}
\end{equation}
where the second inequality holds since we took a null step. 
After $T$ consecutive null steps, the local minorant assumption \eqref{eqn: minorant, ret/vec} ensures  $\modProxGap{k+T} \geq \proxGap{k+T}$. Bounding the number of iterations before this inequality is violated yields an upper bound on $T$.
Thus, solving recurrence \eqref{eqn: simplified master recurrence} by Appendix A of \cite{díaz2023optimal}, recovers the desired bound.

\begin{proof}[Proof of \Cref{lem: master reccurence for null bound}]
Consider a null step $k+1 \leq t \leq k +T$ within a sequence of consecutive non-descent steps. Let $x$ be the fixed proximal center. From \eqref{eqn: subg lower bound, ret/vec} and \eqref{eqn: model subg lower bound, ret/vec}, we define the local lower bound $\twoCut{t+1} \colon T_x \cM \rightarrow \reals$,
\begin{equation*}
    \twoCut{t+1}(v) := \max \left\{f \left( R_x(\canDir{t+1}) \right) + \ip{\vecTObj{t+1}, v-\canDir{t+1}} -\modShift{t+1}, \model{t}(\canDir{t+1}) + \ip{s_{t+1}, v-\canDir{t+1}} \right\}.
\end{equation*}





\begin{lemma}[\cite{díaz2023optimal}, Appendix B]\label{claim: two-cut solution} Consider 
\begin{equation*}
    \twoCutDir  := \operatorname*{argmin}_{v \in T_x \cM} \left\{\twoCut{t+1}(v) + \frac{\rho_{t}}{2}\norm{v}^2\right\}.
\end{equation*}
If $\frac{1}{2}\modProxGap{t} \geq \frac{\modShift{t+1}}{1-\beta}$, then the solution is given by
\begin{align}
    \theta_{t+1}  &= \min \left\{1, \frac{\stepsize{t}\left[f(R_x(\canDir{t+1})) - \modShift{t+1} - \model{t}(\canDir{t+1})\right]}{\norm{\vecTObj{t+1}  - s_{t+1}}^2}\right\}\nonumber \\
    \twoCutDir & = - \frac{1}{\stepsize{t}}(\theta_{t+1}\vecTObj{t+1}  + (1-\theta_{t+1})s_{t+1}). \label{Convex Comb Model and True Subgradient}
\end{align}
\end{lemma}

Let $\canDir{t+2} = \argmin_{v \in T_x \cM} \{\model{t+1}(v) + \frac{\rho_{t}}{2}\norm{v}^2 \}$. Then, 
the model subproblem objective satisfies
\begin{align}
        &\model{t+1}(\canDir{t+2}) + \frac{\rho_{t}}{2}\norm{\canDir{t+2}}^2 \nonumber 
        \geq  \twoCut{t+1}(\canDir{t+2}) + \frac{\rho_{t}}{2}\norm{\canDir{t+2}}^2 \nonumber 
        \geq \twoCut{t+1}(\twoCutDir ) + \frac{\rho_{t}}{2}\norm{\twoCutDir }^2 \nonumber \\
    & \geq \theta_{t+1}\left(f  (R_x(\canDir{t+1})) - \ip{\vecTObj{t+1} , \canDir{t+1}} + \ip{\vecTObj{t+1} , \twoCutDir }-\modShift{t+1}\right) \nonumber \\
    & \quad \quad + (1 - \theta_{t+1})\left(\model{t}(\canDir{t+1}) - \ip{s_{t+1}, \canDir{t+1}} + \ip{s_{t+1}, \twoCutDir }\right) + \frac{\rho_{t}}{2}\norm{\twoCutDir }^2 \nonumber \\
    & = \model{t}(\canDir{t+1}) + \theta_{t+1}\left(f (R_x(\canDir{t+1}) ) -\modShift{t+1} - \model{t}(\canDir{t+1})\right ) \nonumber \\
    & \quad \quad + \ip{\theta_{t+1}\hat{g}_{t+1} + (1-\theta_{t+1}) s_{t+1}, \twoCutDir - \canDir{t+1}} + \frac{\rho_{t}}{2} \norm{\twoCutDir}^{2} \nonumber\\
    \label{Before Manipulation}
    & = \model{t}(\canDir{t+1}) + \theta_{t+1}\left(f (R_x(\canDir{t+1}) ) -\modShift{t+1} - \model{t}(\canDir{t+1})\right ) \nonumber\\
    & \quad \quad - \frac{\theta_{t+1}^2}{2\rho_{t}}\norm{\vecTObj{t+1}  - s_{t+1}}^2 + \frac{\rho_{t}}{2}\norm{\canDir{t+1}}^2. \nonumber
\end{align}
The first inequality follows from $\model{t+1}(\canDir{t+2}) \geq \twoCut{t+1}(\canDir{t+2})$, the second from the definition of $\twoCutDir$, and the third from a convex combination of the cuts of $\twoCut{t+1}$. The last equality uses $s_{t+1} = -\rho_{t}\canDir{t+1}$, definition of $\twoCutDir$ and completing the square. Since $\rho_{t} \leq \rho_{t+1}$, it follows by definition $\modProxGap{t+1}\leq\tilde{\Delta}_{t+1}^{(\rho_{t})}$. Thus, rearranging terms and applying this inequality, we obtain 
\begin{equation*}\label{unrefined recurrence}
    \modProxGap{t+1} \leq \modProxGap{t} - \left[\underbracket{\theta_{t+1}\left(f (R_x(\canDir{t+1}) ) -\modShift{t+1} - \model{t}(\canDir{t+1})\right) - \frac{\theta_{t+1}^2}{2\rho_{t}}\norm{\vecTObj{t+1}  - s_{t+1}}^2}_{T_{1}}\right]. 
\end{equation*}
The amount of decrease can be lower bounded by 
\begin{align*}
    T_{1} & \geq \frac{1}{2}\min \left \{f (R_x(\canDir{t+1})) -\modShift{t+1} - \model{t}(\canDir{t+1}), \frac{\rho_{t}\left(f (R_x(\canDir{t+1}))  -\modShift{t+1} - \model{t}(\canDir{t+1})\right)^2}{\norm{\vecTObj{t+1} -s_{t+1}}^2}\right\} \nonumber \\ 
    & \geq \frac{1}{2}\min \left \{(1-\beta)\modProxGap{t} -\modShift{t+1}, \frac{\rho_{t} \left ((1-\beta)\modProxGap{t} -\modShift{t+1} \right )^2}{\norm{\vecTObj{t+1}  - s_{t+1}}^2}\right\} \nonumber \\
    & \geq \frac{(1- \beta)}{2}\min\left \{\modProxGap{t} -\frac{\modShift{t+1}}{1-\beta}, \frac{\rho_{t}(1-\beta)\left (\modProxGap{t} -\frac{\modShift{t+1}}{1-\beta}\right )^2}{2\norm{\vecTObj{t+1} }^{2} + 2 \norm{s_{t+1}}^2} \right \}
\end{align*}
where the first inequality follows from the definition of $\theta_{t+1}$, the second from failing the criterion \eqref{eqn: descent criterion} and rearrangement. $\tilde{\Delta}_t$ is non-increasing since $\modProxGap{t} -\frac{\modShift{t+1}}{1-\beta}\geq 0$.

Applying the second inequality $\frac{\rho_{t}\norm{\canDir{t+1}}^{2}}{2} \leq \modProxGap{t}
 \leq \frac{\constModelLip{\rho_{0}}^{2}\fLip^2}{2\rho_{k+1}}$ (which we show soon) it follows that $1 \geq \frac{2\stepsize{k+1}\modProxGap{t}}{\constModelLip{\rho_{0}}^{2}\fLip^2} \geq \frac{2\stepsize{k+1}(\modProxGap{t} - \frac{\modShift{t+1}}{1-\beta})}{\constModelLip{\rho_{0}}^{2}\fLip^2}$. Thus
 \begin{align*}
     T_{1} & \geq \frac{(1-\beta)}{2}\min \left\{\frac{2\stepsize{k+1}\left(\modProxGap{t} - \frac{\modShift{t+1}}{1-\beta}\right)^2}{\constModelLip{\rho_{0}}^{2}\fLip^2}, \frac{\rho_{k+1}(1-\beta)\left(\modProxGap{t} -\frac{\modShift{t+1}}{1-\beta}\right)^2}{4\constModelLip{\rho_{0}}^{2}\fLip^2}\right\}\\
     & \geq \frac{(1-\beta)^2\stepsize{k+1}}{8\constModelLip{\rho_{0}}^{2}\fLip^2}\left(\modProxGap{t} - \frac{\modShift{t+1}}{(1-\beta)}\right)^2,
 \end{align*}
where in the first inequality we apply $\norm{\hat{g}_{t+1}} \leq\constModelLip{\rho_{0}}\fLip$, $\norm{s_{t+1}}^{2} \leq \frac{\rho_{t}}{\rho_{k+1}} \constModelLip{\rho_{0}}^{2}\fLip^{2}$ since $\frac{\norm{s_{t+1}}^{2}}{2\rho_{t}} = \frac{\rho_{t}\norm{\canDir{t+1}}^{2}}{2}$, and $\rho_{t} \geq \stepsize{k+1}$. Thus, \eqref{eqn: master recurrence} holds.

The claimed inequality holds because
\begin{align*}
    \frac{\rho_{t}\norm{\canDir{t+1}}^2}{2}
    &\leq \model{t}(0_{x}) - \left(\model{t}(\canDir{t+1}) + \frac{\rho_{t}}{2}\norm{\canDir{t+1}}^2\right)
    \leq 
    \modProxGap{t}
    \leq \modProxGap{k+1} 
    \leq \frac{\norm{g_{k+1}}^2}{2\stepsize{k+1}} 
\end{align*}
where 
we first use the $\rho_t$-strong convexity of $\model{t}(\cdot) + \frac{\rho_{t}}{2}\norm{\cdot}^{2}$ with $\canDir{t+1}$ being the minimizer, then model proximal gaps are non-increasing, then \eqref{eqn: anchor lower bound}. Lastly $\norm{g_{k+1}}\leq \fLip \leq \constModelLip{\rho_{0}}\fLip$, as $h_{b}(\eta) \geq 1$. 
\end{proof}

\subsubsection{Proof of \texorpdfstring{\Cref{lem: stepsize controls descent for recurrence}}{Lemma 4.4}} \label{subsec: proof of uniform lower bound for proximal parameter}


First, note that for $\rho_{t} = \rho$ 
\begin{equation*}
    \tilde{\Delta}_{t}^{(\rho)} \geq \Delta_{t} \geq \begin{cases}
        \frac{1}{2\rho}\left(\frac{\delta_{k+1}}{D_{k+1}}\right)^{2}, \quad &  \delta_{k+1} \leq \rho D_{k+1},\\
        \frac{1}{2}\delta_{k+1}, &  \delta_{k+1} > \rho D_{k+1},
    \end{cases}
\end{equation*}
where $f(v_{\ast}) \geq \hat{f}_{t}(v_{\ast})$ at $v_{\ast} = \operatorname{argmin}_{v}\{f(\exp_{x_{t}}(v))+\frac{\rho}{2}\norm{v}_{2}^{2}\}$ implies the first inequality and the second from \Cref{lem: connect prox step to objective gap}. Let $\tilde{C} = (2\sqrt{-K_{\operatorname{min}}} + \CRet + 2 \CVec)$ and $\modShift{t+1}$ be defined with $\rho$. To show $\rho$ implies \eqref{eqn: sufficient null progress} it suffices to show
\begin{equation} \label{eqn: easier goal for uniform adaptive step}
    \frac{\modShift{t+1}}{1-\beta} \leq \begin{cases}
        \frac{1}{4\rho}\left(\frac{\delta_{k+1}}{D_{k+1}}\right)^{2}, \quad &   \delta_{k+1} \leq \rho D_{k+1},\\
        \frac{1}{4}\delta_{k+1}, & \delta_{k+1} > \rho D_{k+1}.
    \end{cases}
\end{equation}

We consider two cases of $\rho$. First, if $2 \norm{g_{k+1}} \leq \rho$, then $\left(\frac{2 \norm{g_{k+1}}}{\rho}\right)^{2} \leq \frac{2 \norm{g_{k+1}}}{\rho}$. With $\norm{g_{t+1}}, \norm{g_{k+1}} \leq \fLip$, we upper bound the left-hand side of \eqref{eqn: easier goal for uniform adaptive step}:
\begin{equation*}
    \frac{\modShift{t+1}}{1-\beta} = \frac{\tilde{C}\norm{g_{t+1}}}{(1-\beta)} \left[\frac{2\norm{g_{k+1}}}{\rho} + C_{R} \left[\frac{2\norm{g_{k+1}}}{\rho} \right]^2\right]^2 \leq \frac{4\tilde{C}(1+\CRet)^{2}\fLip^{3}}{(1-\beta)\rho^{2}}. 
\end{equation*}
Comparing to the target right-hand side of \eqref{eqn: easier goal for uniform adaptive step}, the inequality is satisfied if
\begin{equation*}
    \rho \geq \begin{cases}
        \frac{16 \tilde{C}(1+ \CRet)^{2} \fLip^{3}}{(1-\beta)}\cdot \left(\frac{D_{k+1}}{\delta_{k+1}}\right)^{2}, \quad & \delta_{k+1} \leq \rho D_{k+1},\\
        \left(\frac{16 \tilde{C}(1+ \CRet)^{2} \fLip^{3}}{(1-\beta)}\right)^{1/2}\delta_{k+1}^{-1/2}, & \delta_{k+1} > \rho D_{k+1}.
    \end{cases}
\end{equation*}
Second, if $0 < \rho < 2\norm{g_{k+1}}$, then $\frac{2\norm{g_{k+1}}}{\rho} < (\frac{2 \norm{g_{k+1}}}{\rho})^{2}$. With $\norm{g_{t+1}}, \norm{g_{k+1}} \leq \fLip$, we upper bound the left-hand side of \eqref{eqn: easier goal for uniform adaptive step}:
\begin{equation*}
    \frac{\modShift{t+1}}{1-\beta} = \frac{\tilde{C}\norm{g_{t+1}}}{(1-\beta)} \left[\frac{2\norm{g_{k+1}}}{\rho} + C_{R} \left[\frac{2\norm{g_{k+1}}}{\rho} \right]^2\right]^2 \leq \frac{4\tilde{C}(1+\CRet)^{2}\fLip^{4}}{(1-\beta)\rho^{3}}. 
\end{equation*}
Again, comparing to the target right-hand side of \eqref{eqn: easier goal for uniform adaptive step}, the inequality is satisfied if
\begin{equation*}
    \rho \geq \begin{cases}
        \left(\frac{16 \tilde{C}(1+ \CRet)^{2} \fLip^{3}}{(1-\beta)}\right)^{1/3}\fLip^{1/3}\left(\frac{D_{k+1}}{\delta_{k+1}}\right)^{2/3}, \quad &  \delta_{k+1} \leq \rho D_{k+1},\\
        \left(\frac{16 \tilde{C}(1+ \CRet)^{2} \fLip^{3}}{(1-\beta)}\right)^{1/4}\fLip^{1/4}\delta_{k+1}^{-1/4}, & \delta_{k+1} > \rho D_{k+1}.
    \end{cases}
\end{equation*}
Taking the max over the lower bounds gives the first claim. The simplification in \eqref{eqn: form for lower bound on proximal parameters, Holder} follows directly from the growth condition, i.e., $D_{k+1}/\delta_{k+1} \leq \mu^{-1/p}\delta_{k+1}^{1/p-1}$.

\subsection{Proof of \texorpdfstring{\Cref{thm: convergence Lipschitz no growth}}{Theorem 1}} \label{subsec: proof of backtracking version}
The proof of the following is omitted, as it closely follows that of \cite{díaz2023optimal} while accounting for our updated null-step bound constants.
\begin{lemma}[Adapted Theorem 2.1, \cite{díaz2023optimal}]\label{lem: constant step bounds}
    Suppose assumptions of \Cref{thm: convergence Lipschitz no growth} hold. Let $f$ be an $\fLip$-Lipschitz g-convex function. Then, running \Cref{alg: Adaptive Riemannian Proximal Bundle Methods} with constant $\rho\geq \rho_{0}$ for some $\rho_{0} > 0$, the number of descent steps before an $\epsilon$-minimizer is found is at most
    \begin{equation*}
        \left \lceil \frac{2}{\beta}\log\left(\frac{f(x_0) - f(x_\ast)}{\rho D^2}\right)\right \rceil_+ + \frac{2 \rho D^2}{\beta \epsilon}  ,
    \end{equation*}
    and the number of null steps is at most 
    \begin{equation*}
        \frac{64 \constModelLip{\rho_{0}}^{2}\fLip^2}{\beta(1-\beta)^2\rho^2 D^2} + \frac{96 \rho \constModelLip{\rho_{0}}^{2}\fLip^{2} D^4}{\beta(1-\beta)^2 \epsilon^3}.
    \end{equation*}
    The first terms of each bound correspond with $\delta_{k} > \rho D^{2}$, the latter with $\epsilon \leq \delta_{k} \leq \rho D^{2}$.
\end{lemma}

Lastly, we introduce the following notation. For fixed $\rho_{k} > 0$ at proximal center $x_{k}$, the bound of \Cref{lem: descent guarantee} can be simplified to only depend on $\delta_{k}$ as
\begin{equation*}
    \Delta_{k} \leq \begin{cases}
        \frac{1}{2\rho_{k}}\left(\frac{\delta_{k}}{D_{k}}\right)^{2}, \quad &\delta_{k} \leq \rho_{k} D^{2},\\
        \frac{1}{2}\delta_{k}, & \delta_{k}>\rho_{k} D^{2}.
    \end{cases}
\end{equation*}
Combining with \Cref{lem: connect prox step to objective gap} yields the following on descent step $k$
\begin{equation} \label{eqn: recurrence on objective gap}
\delta_{k+1} \leq \begin{cases}
    \delta_{k} - \frac{\beta\delta_{k}^{2}}{2\rho_{k} D^{2}}, \quad & \delta_{k} \leq \rho_{k} D^{2},\\
    \left( 1 - \frac{\beta}{2} \right) \delta_{k}, & \delta_{k}>\rho_{k} D^{2}.
\end{cases}
\end{equation}
As $\delta_{k}$ decreases and $\rho_{k}$ is non-decreasing (by doubling), the condition $\delta_{k} \leq \rho_{k}D^{2}$ eventually triggers. Thus, the algorithm transitions from the second case to the first case at most once. We denote proximal parameters based on how many updates occur: $\rho_{k_{0}}^{(0)} < \rho_{k_{1}}^{(1)} < \dots < \rho_{k_{s-1}}^{(s-1)} <  \rho_{k_{s}}^{(s)} < \rho_{k_{s+1}}^{(s+1)} < \dots < \rho_{k_{J}}^{(J)} < 2\tilde{\rho}.$ 
The iteration where the $j$th update occurs is denoted by $k_j$. The $s$th update indicates when the algorithm transitions between cases in \eqref{eqn: recurrence on objective gap}. We now proceed to the proof.

\textbf{Bounding Doubling Steps.} Bound directly follows from \Cref{lem: uniform lower bound for proximal}.

\textbf{Bounding Descent and Null Steps.} To establish the total iteration complexity, we analyze each case of the recurrence \eqref{eqn: recurrence on objective gap}.

\textbf{Case 1. $\delta_k > \rho^{(s)}D^2$:} Let $T^{(1)}$ and $N^{(1)}$ denote the number of descent and null steps in Case 1, respectively. Since descent steps result in a $(1-\beta/2)$ contraction, the number of descent steps until $\delta_{k} \leq \rho^{(s)}D^{2}$ is bounded by 
  $  T^{(1)} \leq \left\lceil  \frac{2}{\beta}\log\left( \frac{\delta_{0}}{\rho_{0} D^{2}} \right)  \right\rceil_{+}.$


Let $N^{(1)} = \sum_{j=0}^{s} N_{j}$, where $N_{j}$ denotes the number of null steps associated with $\rho^{(j)}$. Application of \Cref{lem: constant step bounds} to each $N_{j}$, and $\rho^{(j)} = 2^{j}\rho_{0}$ results in
\begin{equation*}
    N^{(1)} \leq \sum_{j=0}^s \frac{64 \constModelLip{\rho_{0}}^{2}\fLip^2}{\beta(1-\beta)^{2} [\rho^{(j)}]^{2} D^{2}} = \frac{64 \constModelLip{\rho_{0}}^{2}\fLip^2}{\beta(1-\beta)^{2} \rho_{0}^{2} D^{2}}  \sum_{j=0}^s 2^{-2j} \leq \frac{256 \constModelLip{\rho_{0}}^{2}\fLip^2}{3\beta(1-\beta)^{2}\rho_{0}^{2}D^{2}}.
\end{equation*}

\textbf{Case 2.  $\delta_k \leq \rho^{(s)}D^2$:} Let $T^{(2)}$ be the number of descent steps in this case. Notice \eqref{eqn: recurrence on objective gap} ensures $\delta_{k+1} \leq \delta_{k} - \frac{\beta\delta_{k}^{2}}{2 \rho_{k}D^{2}}$, which rearranges to $\frac{1}{\delta_{k+1}} - \frac{1}{\delta_{k}} \geq \frac{\beta}{2\rho_{k}D^{2}} \geq \frac{\beta}{2\rho^{(J)}D^{2}}$ as $\rho_{k} \leq \rho^{(J)}$. Since after the final step of this case $\delta_{\operatorname{final}}  \leq \epsilon$, we have
\begin{equation*}
    \frac{1}{\epsilon} > \frac{1}{\delta_{\operatorname{final}}} - \frac{1}{\delta_{\operatorname{start}}} \geq \sum_{m=1}^{T^{(2)}} \frac{\beta}{2\rho^{(J)}D^{2}} \geq T^{(2)}\frac{\beta}{2\rho^{(J)}D^{2}} 
\end{equation*}
where $\delta_{\operatorname{start}}$ is the gap before descent steps. Rearrangement recovers $T^{(2)} < \frac{2 \rho^{(J)}D^{2}}{\beta \epsilon}$.

Second, let $N^{(2)} = \sum_{j=s}^{J}N_{j}$, where $N_{j}$ denotes the number of null steps associated with $\rho^{(j)}$. Application of \Cref{lem: constant step bounds} to each $N_{j}$, and $\rho^{(j)} = 2^{-(J-j)}\rho_{0}$ results in
\begin{equation*}
    N^{(2)} \leq 
    \frac{96 \rho^{(J)} \constModelLip{\rho_{0}}^{2}\fLip^2D^{4}}{\beta(1-\beta)^{2} \epsilon^{3}}  \sum_{j=0}^s 2^{-(J-j)} \leq
    \frac{192 \rho^{(J)} \constModelLip{\rho_{0}}^{2}\fLip^2D^{4}}{\beta(1-\beta)^{2} \epsilon^{3}}.
\end{equation*}
Finally, $\rho^{(J)}\leq \max\{\rho_{0},2\tilde{\rho}\}$, by doubling procedure, yields the claimed bound.

\subsection{Proof of \texorpdfstring{\Cref{thm: convergence rate with growth tuned}}{Theorem 2}}\label{subsec:proof-of-improved-convergence}

There are four distinct phases during an algorithm run with this schedule. \Cref{lem: tuned proximal parameter sequence implication} identifies the last phase, and bounds total iterations in earlier phases. The proof is deferred to \Cref{appendix: proof of tuned proximal parameter sequence implication}.
\begin{lemma} \label{lem: tuned proximal parameter sequence implication}
    Suppose assumptions of \Cref{thm: convergence rate with growth tuned} hold. Let \Cref{alg: Adaptive Riemannian Proximal Bundle Methods} run with proximal parameter schedule \eqref{eqn: specialized proximal parameter schedule} until an $\epsilon$-minimizer. Then the following hold. First, $\delta_{k}^{1-2/p} \leq \rho_{k}/\mu^{2/p}$ for all $\delta_{k} := f(x_{k}) - f_{\ast}$, and $p \geq 1$. Second, let
    \begin{equation*}
        \swapGap := \begin{cases}
            \min\left\{\left(\frac{A\fLip}{\mu^{8/p}}\right)^{p/(4-3p)}, \left(\frac{A}{\mu^{8/p}}\right)^{p/(8-5p)}, \left(\frac{A}{\mu^{4/p}}\right)^{p/(3p-4)}\right\}, \quad &p \in [1,4/3),\\
            \min\left\{\left(\frac{\max\{A \mu^{-3/2}, A^{1/2}, \mu^{3/2}\}}{A^{1/3}\fLip^{1/3}\mu^{-1/2}}\right)^{3}, \left(\frac{\max\{A \mu^{-3/2}, A^{1/2}, \mu^{3/2}\}}{A^{1/4}}\right)^{4}\right\}, & p = 4/3,\\
            \min\left\{\left(\frac{A^{2}}{\fLip\mu^{4/p}}\right)^{\frac{p}{4(p-1)}}, \left(\frac{A^{3}}{\mu^{8/p}}\right)^{p/(8p-6)}, \left(\frac{A}{\mu^{4/p}}\right)^{p/(3p-4)}\right\}, & p > 4/3,
        \end{cases}
    \end{equation*}
    when $A:= \frac{16 (2\sqrt{-K_{\operatorname{min}}} +2 \CVec + \CRet)(1+\CRet)^{2} G_{f}^{3}}{1-\beta} > 0$. Then, for $\delta_{k} \leq \swapGap$
    \begin{equation*}
        \rho_{k} = \begin{cases}
            \mu^{2/p}\delta_{k}^{(p-2)/p}, &p \in [1,4/3),\\
            \max\{A\mu^{-3/2}, A^{1/2}, \mu^{3/2}\} \delta_{k}^{-1/2}, & p = 4/3, \\
             A\mu^{-2/p}\delta_{k}^{(2-2p)/p}, &p > 4/3.
        \end{cases}
    \end{equation*}
    Lastly, when $\delta_{k} >\swapGap > \epsilon$ the number of descent steps is bounded by
    \begin{align}
    \label{eqn: bound on descent steps initial stage}
        T_{\operatorname{bound}} := \begin{cases}
            \left\lceil T_{1} + \frac{2\noDoubleConst\left(\delta_{0} - \swapGap \right)}{\beta \mu^{4/p}\swapGap^{(4p-4)/p}} + \frac{2A^{1/2}(\delta_{0} - \swapGap)}{\beta \mu^{2/p} \swapGap^{(5p-4)/p}}\right\rceil_{+} , \quad & p \in [1, 4/3),\\
            \left\lceil T_{1}\right\rceil_{+}, & p =  4/3,\\
            \left\lceil T_{1} +  \frac{2A^{1/2}(\delta_{0} - \swapGap)}{\beta \mu^{2/p} \swapGap^{(5p-4)/p}} + \frac{2}{\beta}\log \left( \frac{\delta_{0}}{\swapGap} \right)   \right\rceil_{+}, & p > 4/3,
        \end{cases}
    \end{align}
    with $T_{1} := \frac{2\noDoubleConst^{1/3}\fLip^{1/3}\left(\delta_{0} - \swapGap \right)}{\beta \mu^{2/3p}\swapGap^{(8p-8)/3p}} + \frac{2\noDoubleConst^{1/4}G_{f}^{1/4}\left(\delta_{0} - \swapGap \right)}{\beta \mu^{2/p}}$. The number of null steps is bounded by
    \begin{align}
    \label{eqn: bound on null steps initial stage}
        \begin{cases}
            \frac{64 (\constModelLip{\rho_{0}})^{2} \fLip^{2}}{(1-\beta)^{2}\mu^{2/p} \swapGap^{2-2/p}}(T_{\operatorname{bound}} ), \quad & p \in [1,4/3],\\
            \frac{64 (\constModelLip{\rho_{0}})^{2} \fLip^{2}}{(1-\beta)^{2}\mu^{2/p} \swapGap^{2-2/p}}(T_{\operatorname{bound}} ) + \left( \frac{1}{1-(1-\beta/2)^{2-2/p}} \right) \frac{64(\constModelLip{\rho_{0}})^{2}G_{f}^{2}}{(1-\beta)^{2}\mu^{2/p}\swapGap^{2-2/p}}, \quad & 
            p > 4/3.
        \end{cases}
    \end{align}
\end{lemma}

To establish convergence rates, \Cref{lem: tuned proximal parameter sequence implication} reduces the analysis to bounding the number of descent and null steps, denoted $T$ and $N$ respectively, when $\epsilon < \delta_{k} < \swapGap$. We recover explicit bounds for three cases of $p \geq 1$ when $\epsilon < \delta_{k} < \swapGap$. Computing the oracle complexities of these bounds recover the stated rates.

\subsubsection{Case 1: \texorpdfstring{$p < 4/3$}{p < 4/3}} Then $\rho_{k} = \mu^{2/p}\delta_{k}^{(p-2)/p}$, which matches Theorem 2.6 of \cite{díaz2023optimal}; analogous arguments yield
\begin{align*}
    T &\leq \left\lceil  \frac{2}{\beta}\log \left( \frac{\swapGap}{\epsilon} \right)   \right\rceil _{+},  
    \qquad N \leq \begin{cases}
\left( \frac{1}{1-(1-\beta/2)^{2-2/p}} \right) \frac{64 (\constModelLip{\rho_{0}})^{2}G_{f}^{2}}{(1-\beta)^{2}\mu^{2/p}\epsilon^{2-2/p}}, & \text{if } p > 1,\\
\frac{64 (\constModelLip{\rho_{0}})^{2}\fLip^{2}}{(1-\beta)^{2}\mu^{2}}\left \lceil \frac{2}{\beta} \log\left(\frac{\swapGap}{\epsilon}\right)\right \rceil_{+}, & \text{if } p = 1.
\end{cases}
\end{align*}

\subsubsection{Case 2: \texorpdfstring{$p = 4/3$}{p = 4/3}} Then, $\rho_{k} = \max\{A\mu^{-3/2}, A^{1/2}, \mu^{3/2}\}\delta_{k}^{-1/2}$. With $\lambda := \max\{A\mu^{-3/2}, A^{1/2}, \mu^{3/2}\}$, the schedule corresponds with Theorem 2.6 of \cite{díaz2023optimal} (up to a constant factor), and analogous arguments recover
\begin{equation*}
    T \leq \left\lceil  \frac{2\lambda}{\beta \mu^{3/2}}\log \left( \frac{\delta_{0}}{\epsilon} \right)   \right\rceil_{+},  \qquad  N \leq \left( \frac{1}{1-(1 -\frac{\beta \mu^{3/2}}{2 \lambda})^{1/2}} \right)\frac{64 (\constModelLip{\rho_{0}})^{2} G_{f}^{2}}{(1-\beta)^{2}\mu^{3/2}\epsilon^{1/2}}.
\end{equation*}

\subsubsection{Case 3: \texorpdfstring{$p > 4/3$}{p > 4/3}} Note
$\rho_{k} = \noDoubleConst\mu^{-2/p}\delta_{k}^{(2-2p)/p}$ implies $\delta_{k}^{1-2/p} \leq \rho_{k}/\mu^{2/p}$, by \Cref{lem: tuned proximal parameter sequence implication}.
Then, we have $\Delta_{k} \geq \mu^{2/p}\delta_{k}^{2-2/p}/2\rho_{k} = \mu^{4/p}\delta_{k}^{(4p-4)/p}/2\noDoubleConst$ using \Cref{lem: connect prox step to objective gap} and $p$-H\"{o}lder growth. This lower bound with \Cref{lem: descent guarantee} implies $\delta_{k+1} \leq \delta_{k} - \frac{\beta \mu^{4/p}}{2\noDoubleConst}\delta_{k}^{4-4/p}$. Solving the recurrence yields
\begin{equation*}
     T\leq \left \lceil \frac{2 \noDoubleConst p}{(3p-4)\beta \mu^{4/p} \epsilon^{3 - 4/p}} \right \rceil_+ 
\end{equation*}
by applying Appendix A of \cite{díaz2023optimal} since $4-4/p > 1$. 

Combining the first lower bound on $\Delta_{k}$ with \Cref{lem: null step bound imperfect info} bounds the number of consecutive null steps as follows: $\frac{32 (\constModelLip{\rho_{0}})^{2} \fLip^{2}}{(1-\beta)^{2}\rho_{k} \Delta_{k}} \leq \frac{64 (\constModelLip{\rho_{0}})^{2} \fLip^{2}}{(1-\beta)^{2}\mu^{2/p}\delta_{k}^{2-2/p}} \leq \frac{64 (\constModelLip{\rho_{0}})^{2} \fLip^{2}}{(1-\beta)^{2}\mu^{2/p}\epsilon^{2-2/p}}$.
As the bound is independent of $k$, multiplying by the bound on $T$ yields
\begin{equation*}
     N \leq \frac{128 \noDoubleConst p (\constModelLip{\rho_{0}})^{2} \fLip^2}{(3p-4)\beta(1-\beta)^{2}\mu^{6/p} \epsilon^{5-6/p}}.
\end{equation*}

\section{Numerical Experiments}\label{section: numerical experiments}
In this section, we present numerical results that support our theoretical guarantees. \Cref{subsec: riemannian median} and \Cref{subsec: denoising hyperbolic} detail finding the Riemannian median and denoising a hyperbolic signal through total variation, respectively. The code for reproducing these experiments is available at 

\begin{center}
\textcolor{myPurple}{https://github.com/mcphersonianoliver/Riemannian-Proximal-Bundle-Method}.
\end{center}
\paragraph{Implementation details} Experiments were conducted using \texttt{Julia v1.12.1} on a MacBook Air (M2, 2022) with 8 GB of RAM, running macOS 12.5. The manifold primitives were implemented utilizing \texttt{Manopt.jl} \cite{manopt}.
We implement our Riemannian proximal bundle method (RPB) 
using the three-cut model defined by \eqref{eqn: subg lower bound, ret/vec}, \eqref{eqn: model subg lower bound, ret/vec}, and  \eqref{eqn: anchor lower bound} of \Cref{assump: imperfect setting}
\begin{equation} \label{eqn: three-cut model}
    \model{i+1}(v) := \max\left\{\ell_{\operatorname{new}}(v), \ell_{\operatorname{agg}}(v), \ell_{\operatorname{anchor}}(v)\right\}.
\end{equation}
This minimal model serves two purposes: $(i)$ it demonstrates theoretical efficiency independent of the specific model construction $\model{i}$ and $(ii)$ enables solving the proximal subproblem \eqref{eqn: candidate direction subproblem} analytically via an exhaustive check of all cases. We mention in passing that $\ell_{\operatorname{new}}(d)$ is always active, reducing the number of cases to four. As a default, we set $\rho_{0} = 1.0$ unless stated otherwise.

We compare our implementation against three first-order methods in the literature. Two of these are proximal bundle methods: the proximal bundle algorithm (PBA) introduced in \cite{hoseini2023proximal} and the Riemannian convex bundle method (RCBM) introduced in \cite{bergmann2025convexbundle}. The last compared algorithm is the subgradient method (SGM) introduced in \cite{ferreira1998subgradient}. 
All of these methods are available in \texttt{Manopt.jl} and are used with their default parameters when available, except when stated otherwise. The alternative proximal bundle methods require solving QP problems  for their proximal subproblems, which are solved using  \texttt{RipQP.jl}  \cite{orban_ripqp_2020}. 
\subsection{Riemannian Median} \label{subsec: riemannian median}
Consider $\symMat^d_+ = \{\matr{X} \in \symMat^d : \matr{X} \succ0\}$,
where $\symMat^{d}$ is the set of real $d \times d$ symmetric matrices. The set $\symMat^{d}_+$, endowed with the  \textit{affine invariant metric} $\ip{\xi_\matr{X}, \eta_\matr{X}}_\matr{X} = \operatorname{tr}(\matr{X}^{-1}\xi_\matr{X} \matr{X}^{-1} \eta_\matr{X})$, where $\xi_\matr{X}, \eta_\matr{X} \in T_{\matr{X}}\symMat^d_+ \cong \symMat^d$,
is a \textit{Hadamard manifold}. 
Let $\exp_\matr{X} \colon T_{\matr{X}}\symMat^d_{+} \rightarrow \symMat^d_{+}$  be defined by $\exp_{\matr{X}}(\xi) = \matr{X}^{\frac{1}{2}}\exp(\matr{X}^{-\frac{1}{2}} \xi \matr{X}^{-\frac{1}{2}}) \matr{X}^{\frac{1}{2}},$
where $\exp$ is the matrix exponential. The map $\ParT{\matr{Y}}{\matr{X}} \colon T_{\matr{X}}\symMat^d_+ \rightarrow T_{\matr{Y}}\symMat^d_+$ given by $\ParT{\matr{Y}}{\matr{X}}(\xi_{\matr{X}}) = \matr{X}^{\frac{1}{2}}\exp(\frac{\matr{X}^{-\frac{1}{2}}\eta_{\matr{X}} \matr{X}^{-\frac{1}{2}}}{2})\matr{X}^{-\frac{1}{2}} \xi_{\matr{X}} \matr{X}^{-\frac{1}{2}}\exp(\frac{\matr{X}^{-\frac{1}{2}} \eta_{\matr{X}} \matr{X}^{-\frac{1}{2}}}{2})\matr{X}^{\frac{1}{2}}$, where $\eta_{\matr{X}} = \log_{\matr{X}}(\matr{Y})$, is the parallel transport.
Both expressions require computing a matrix exponential, which is computationally expensive. Consider the \textit{first-order} retraction
\begin{equation} \label{eqn: retractions symmetric matrices}
    R_\matr{X}(\xi_\matr{X}) = \matr{X} + \xi_\matr{X}, 
\end{equation}
which arises from the first-order approximation of the matrix exponential. To ensure this is still positive definite, step-size is reduced (or proximal parameter is increased) when necessary.
Additionally, we use the projection vector transport \eqref{eqn: projection transport}, which reduces to the identity map.
For more details, see Section 4.1 of \cite{jeuris2012survey}. 

Given $\{\matr{X}_j\}_{j=1}^n \subseteq \mathbb{S}_{+}^{d}$, consider minimizing $f \colon \mathbb{S}_{+}^{d} \rightarrow \mathbb{R}$ defined by
\begin{equation}\label{eqn: riemannian median objective}
    f(\matr{X}) := \frac{1}{n} \sum_{j=1}^n d_{\mathbb{S}_{+}^{d}}(\matr{X},\matr{X}_j).
\end{equation}
The minimizer is called the \textit{Riemannian median}.


For $d = 55$ we generate a sample of $20$ random data points. We run four algorithms: the three proximal bundle algorithms RPB, RCBM, and PBA, and the subgradient method SGM. We run a version with access to exponential maps and another with first-order retractions for each. All bundle method runs use the projection vector transport. For the implementation of RPB, we set the trust parameter $\beta = 0.1$. For SGM, we use a geometrically decaying stepsize of $\eta_{i} = Cq^{i}$, with $C = 2$ and $q =0.95$ were selected via a grid search over $C \in \{0.5, 1.0, 2.0, 5.0, 10.0, 20.0, 50.0\}$ and $q \in \{0.55, 0.60, 0.65, 0.70, 0.75, 0.80, 0.85, 0.90, 0.95\}$.

    



\begin{figure}[tbp]
    \centering
    \begin{minipage}{0.5\textwidth}
        \centering
        \includegraphics[width=\linewidth]{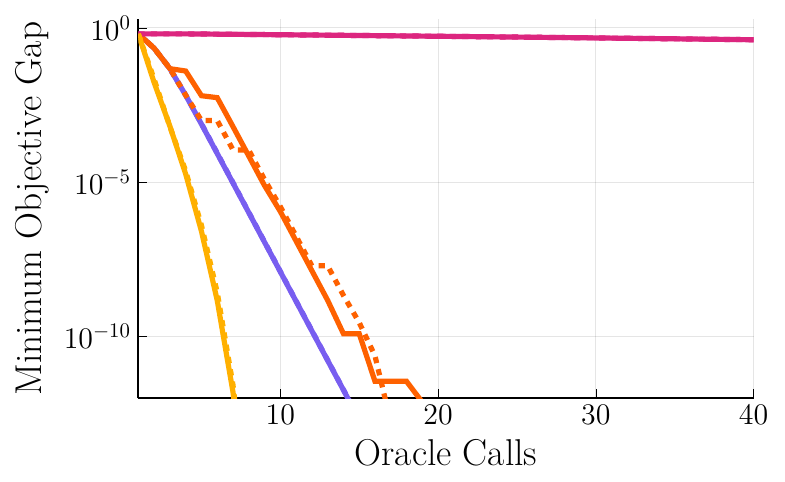}
        \label{fig:objective gap median}
    \end{minipage}
    \hfill 
    \begin{minipage}{0.49\textwidth}
        \centering
        \includegraphics[width=\linewidth]{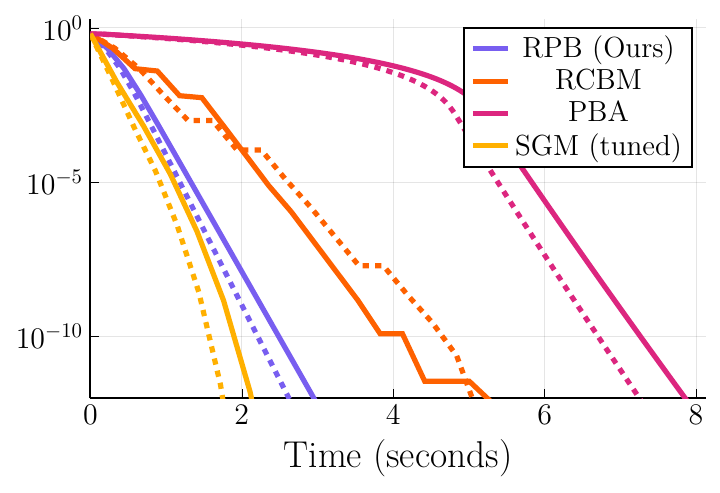}
        \label{fig:Wall-Clock-median}
    \end{minipage}
    \caption{Minimizing Riemannian median objective \eqref{eqn: riemannian median objective} for $\mathbb{S}_{+}^{55}$, for 20 random data points. Solid lines correspond with using exponential maps, the dashed with the first order retraction \eqref{eqn: retractions symmetric matrices}. Projection transports are used by all bundle methods.}
\end{figure}
Figure 1 shows that all methods achieve linear convergence rates. Notably, SGM required extensive hyperparameter tuning while all the bundle methods use default parameters. Regarding wall-clock performance, employing first-order retractions instead of the exponential maps results in faster computations without a heavy penalty in oracle call complexity, reflecting their practicality.
\subsection{Denoising Through Total Variation} \label{subsec: denoising hyperbolic}

Consider the $d$-dimensional hyperbolic space defined by $\mathbb{H}_{d} := \{ x\in \mathbb{R}^{d+1} \colon \ip{x,x}_{\mathbb{H}_{d}} = -1,  x_{d+1} > 0\}$ 
with the \textit{Minkowski metric} $\ip{\cdot, \cdot}_{\mathbb{H}_{d}} \colon \mathbb{H}_{d} \times \mathbb{H}_{d} \to \mathbb{R}$ defined by $\ip{x,y}_{\mathbb{H}_{d}} = \sum_{i=1}^{d} x_{i}y_{i} - x_{d+1}y_{d+1}.$


The exponential map $\exp_{x}(v) \colon T_{x} \mathbb{H}_{d} \to \mathbb{H}_d$ is defined by
\begin{equation*}
    \exp_{x}(v) := \begin{cases}
    \cosh(\norm{v})x + \sinh(\norm{v}) \frac{v}{\norm{v}}, \quad & \norm{v} \neq 0,\\
    x, & \norm{v} = 0,
    \end{cases}
\end{equation*}
where $\norm{v} = \sqrt{\ip{v,v}_{\mathbb{H}_{d}}}$. Parallel transport  $\ParT{y}{x} \colon T_{x} \mathbb{H}_{d} \to T_{y} \mathbb{H}_{d}$ is defined by
 $   \ParT{y}{x}(v) = v + \frac{\ip{v,y}_{\mathbb{H}_d}}{1 - \ip{x,y}_{\mathbb{H}_{d}}}(x + y),
$ 
where addition is performed in the ambient space $\mathbb{R}^{d+1}$. Both operations are computationally efficient.


Let $q \colon \{1,\dots, n\} \rightarrow \mathbb{H}_{2}$ be the discretization of an unknown manifold-valued function $\tilde{q} \colon [a,b] \rightarrow \mathbb{H}_{2}$, where $\{1, \dots, n\}$ indexes a uniform partition of $[a,b]$. Let $\hat{q}$ be a noisy version of $q$. Note that $q, \hat{q} \in \mathcal{M} = (\mathbb{H}_{2})^{n}$.  For a given noisy observation $\hat{q}$ and regularization parameter $\alpha > 0$, we denoise the signal by minimizing the objective $f_{\hat{q}} \colon \cM \rightarrow \reals$ defined by:
\begin{equation}\label{eqn: denoising optimization}
    f_{\hat{q}}(p) := \frac{1}{n}\left(\sum_{i=1}^{n} \frac{1}{2}d_{\mathbb{H}_{2}}^{2}(p^{[i]}, \hat{q}^{[i]}) + \alpha \operatorname{TV}(p)\right),  \quad \operatorname{TV}(p) := \sum_{i=1}^{n-1} d_{\mathbb{H}_{2}}(p^{[i]}, p^{[i+1]}).
\end{equation}


For this numerical experiment, to produce $q$ we embed a square wave into $\mathbb{H}_{2}$ and discretize it into 496 points, following the procedure in Section 6.2 of \cite{bergmann2025convexbundle}. Then, $\hat{q}$ is obtained by perturbing the signal $q$ by mapping isotropic Gaussian noise to the tangent spaces of each point, i.e. $\hat{q}^{[i]} := \operatorname{exp}_{q^{[i]}}(X_{i})$ where $X_i \in T_{q^{[i]}}\mathbb{H}_{2}$ is a random tangent vector whose coordinates, with respect to an orthonormal basis of the tangent space, follow $\mathcal{N}(0, \sigma^{2} I_{2})$ with $\sigma = 0.3$ for each $i \in \{1, \dots, 496\}$. 

\begin{figure}[tbp]
    \centering
    \begin{minipage}{0.49\textwidth}
        \centering
        \includegraphics[width=\linewidth]{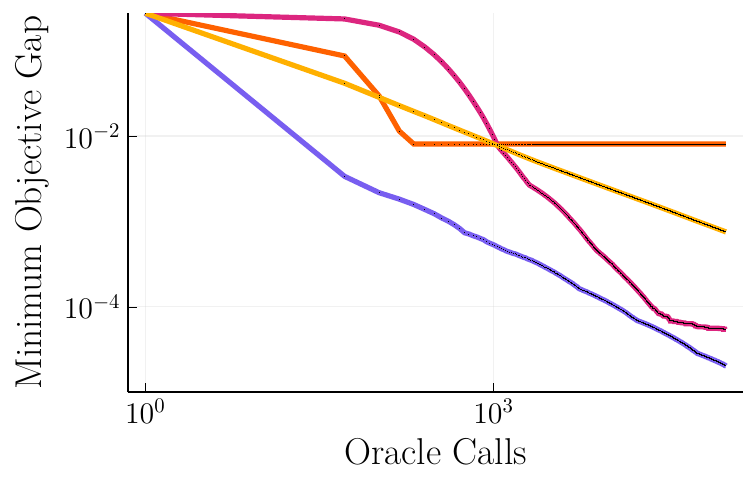}
        \label{fig:objective gap denoise}
    \end{minipage}
    \hfill 
    \begin{minipage}{0.49\textwidth}
        \centering
        \includegraphics[width=\linewidth]{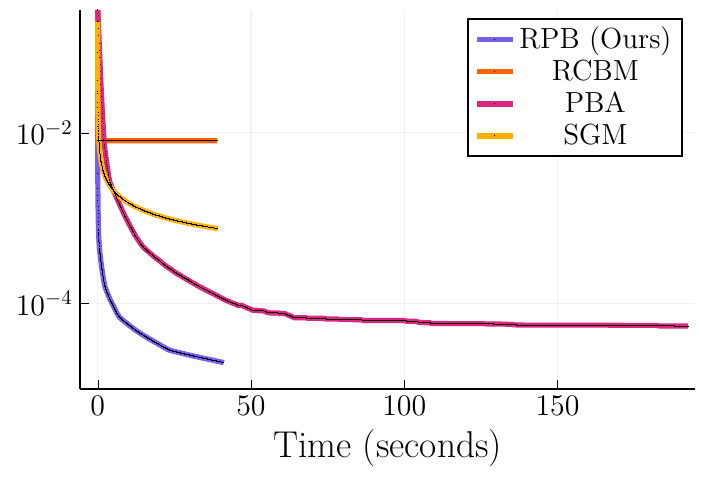}
        \label{fig:Wall-Clock-denoise}
    \end{minipage}

    \caption{Minimizing total variation denoising objective \eqref{eqn: denoising optimization}. All algorithms have access to exponential maps and parallel transports, which are easily computable.}
\end{figure}



We set $\alpha = 0.5$, and run four algorithms for 100,000 iterations: the three proximal bundle methods (RPB, RCBM, and PBA), alongside the subgradient method (SGM). We provide access to exponential maps and parallel transport, composing them with a projection to ensure manifold membership in the presence of numerical errors. When implementing RPB, we set $\rho_{0} = 1.0$ and $\beta = 0.001$. For SGM, we used stepsize $\eta_{k} = k^{-1/2}$. We utilize the cyclic proximal point algorithm (CPPA) \cite{Bacak-CPPA} solely to compute the approximate minimizer $p_{\ast}$ used for error evaluation as it requires a proximal oracle, a stronger mathematical assumption than the subgradient oracle.

As illustrated in Figure 2, RPB, PBA, and SGM exhibit similar sublinear rates of convergence. 
Furthermore, the final proximal parameter for RPB stabilized at approximately 67. This empirical stability demonstrates that the proximal parameter does not reach the pessimistic worst-case theoretical upper bound $\mathcal{O}(\epsilon^{-2})$ (see \Cref{lem: uniform lower bound for proximal}). Finally, we observe that RCBM stalls due to the pruning of cuts when the dual variables fall below the default numerical tolerance of $10^{-8}$. Specifically, repeated QP solves with the transported cuts yielded dual solutions below this threshold, preventing the method from making further progress.


\bibliographystyle{siam}
\bibliography{refs}

\appendix


\section{Proof of Lemma~\ref{lem: error adjustment shift}}\label{appendix: proof of error adjustment}
The following lemma bounds the error arising from sectional curvature when using parallel transport. 
\begin{lemma}\label{lem: parallel transport inherent error}
Suppose $\mathcal{M}$ is Hadamard with sectional curvature bounded below, $-\infty < K_{\operatorname{min}} \leq 0$. Fix $x \in \cM$ and $\alpha > 0$. For any $y, z \in B_{\mathcal{M}}(x,\alpha)$ and $g \in T_{z}\mathcal{M}$:
\begin{equation*}
    |\ip{\ParT{x}{z}[g], \ParT{x}{z}[\log_{z}(y)] - [\log_{x}(y) - \log_{x}(z)]}_{x}| \leq 2 \sqrt{-K_{\operatorname{min}}}\norm{g}_{z} \alpha^{2}.
\end{equation*}
\end{lemma}

\begin{proof}
    The argument follows Alimisis' argument in Appendix C of \cite{alimisis2021momentumimprovesoptimizationriemannian} closely; it is included for completeness. Let $\gamma \colon [0,1] \to \cM$ be the geodesic such that $\gamma(0) = z$ and $\gamma(1) = x$. Denote $\tilde{g} := \ParT{x}{z}[g]$. Consider $h(t) := \langle \tilde{g}, \ParT{x}{\gamma(t)}[\log_{\gamma(t)}(y)] \rangle_{x}$, then $h(0) = \ip{\tilde{g}, \ParT{x}{z}[\log(y)]}_{x}$ and $h(1) = \ip{\tilde{g},\log_{x}(y)}_{x}$. By the Mean Value Theorem, there exists $t_{0} \in (0,1)$ such that $h(1) - h(0) = h'(t_0)$.

    To compute $h'(t)$, we use the fact that parallel transport acts as an isometry and commutes with the covariant derivative along $\gamma:$
    \begin{align*}
        h'(t) & = \left\langle \tilde{g}, \ParT{x}{\gamma(t)}\nabla_{\dot{\gamma}(t)}\log_{\gamma(t)}(y)\right \rangle_{x}\\
        & = \left \langle \tilde{g}, \ParT{x}{\gamma(t)}\operatorname{Hess}_{\gamma(t)}\left(-\frac{1}{2}d_{\cM}^{2}(\cdot, y)\right)[\dot{\gamma}(t)]\right \rangle_{x}\\
        & = - \left \langle \tilde{g}, \ParT{x}{\gamma(t)}\operatorname{Hess}_{\gamma(t)}\left(\frac{1}{2}d_{\cM}^{2}(\cdot, y)\right)[\dot{\gamma}(t)]\right \rangle_{x}.
    \end{align*}
    Since $\gamma(0) = z$ and $\gamma(1) = x$, the velocity vector at $t$ is $\dot{\gamma}(t)=-\ParT{\gamma(t)}{x}[\log_{x}(z)]$. Substituting this into the derivative $h'(t_{0})$ yields
    \begin{equation*}
        h'(t_{0}) = \left\langle \tilde{g}, \underbracket{\left(\ParT{x}{\gamma(t_{0})}\operatorname{Hess}_{\gamma(t_{0})}\left(\frac{1}{2}d_{\cM}^{2}(\cdot, y)\right)\ParT{\gamma(t_{0})}{x}\right)}_{:= \mathcal{H}_{t_{0}}}[\log_{x}(z)]\right\rangle.
    \end{equation*}
    Given that $h(0) - h(1) = -h'(t_{0})$, we have
    \begin{align*}
        |\ip{\tilde{g}, \ParT{x}{z}[\log(y)] - (\log_{x}(y) - \log_{x}(z))}_{x}| & = |h(0) - h(1) + \ip{\tilde{g},\log_{x}(z)}_{x}|\\
        & = |\ip{\tilde{g}, (I - \mathcal{H}_{t_{0}})[\log_{x}(z)]}_{x}|\\
        & \leq \norm{g}_z \norm{I - \mathcal{H}_{t_{0}}}_{\operatorname{op}}\norm{\log_{x}(z)}_{x}.
    \end{align*}
    Because $\cM$ is Hadamard with sectional curvature bounded below by $K_{\min}\leq 0$, it is shown by \cite{alimisis2021momentumimprovesoptimizationriemannian} Appendix D that the eigenvalues of $\operatorname{Hess}_{\gamma(t_{0})}(\frac{1}{2}d_{\mathcal{M}}^{2}(\cdot, y))$ are contained in the interval $\left[1, \sqrt{-K_{\operatorname{min}}}d_{\cM}(\gamma(t_{0}), y) \operatorname{coth}\left(\sqrt{-K_{\operatorname{min}}} d_{\cM}(\gamma(t_{0}), y)\right)\right]$. Since $\ParT{x}{\gamma(t_{0})}$ is an isometry, $\mathcal{H}_{t_{0}}$ shares these eigenvalues. Using the inequality $u\coth (u) - 1 \leq u$ for $u \geq 0$, it follows that
    \begin{equation*}
        \norm{I-\mathcal{H}_{t_{0}}}_{\operatorname{op}} \leq \sqrt{-K_{\operatorname{min}}}d_{\cM}(\gamma(t_{0}), y).
    \end{equation*}
    Finally, since $d_{\mathcal{M}}(\gamma(t_{0}),y) \leq d_{\mathcal{M}}(\gamma(t_{0}), x) + d_{\mathcal{M}}(x,y) \leq 2\alpha$,
    \begin{equation*}
        \norm{g}_z \norm{I - \mathcal{H}_{t_{0}}}_{\operatorname{op}}\norm{\log_{x}(z)}_{x} \leq 2\sqrt{-K_{\operatorname{min}}}\norm{g}_z \alpha \norm{\log_{x}(z)}_{x} \leq 2\sqrt{-K_{\operatorname{min}}}\norm{g}_z \alpha^{2},
    \end{equation*}
    where the last inequality follows from $\norm{\log_{x}(z)} = d_{\mathcal{M}}(x,z) \leq \alpha$.
\end{proof}
    
    We now prove Lemma~\ref{lem: error adjustment shift}. 
    Denote $y := \exp_{x}(v)$. By the subgradient inequality, 
    \begin{align*}
        f(y) & \geq f(z) + \ip{g, \log_{z}(y)}_{z} = f(z) + \ip{\ParT{x}{z}[g], \ParT{x}{z}[\log_{z}(y)]}_x\\
        & = f(z) + \ip{\VecT{x}{z}[g], v - v_{z}}_{x} + \underbracket{\ip{\ParT{x}{z}[g] - \VecT{x}{z}[g], v - v_{z}}_x}_{T_{1}} \\&+ \underbracket{\ip{\ParT{x}{z}[g], v_{z} - \log_{x}(z)}_{x}}_{T_{2}}  + \underbracket{\ip{\ParT{x}{z}[g], \ParT{x}{z}[\log_{z}(y)] - [\log_{x}(y)- \log_{x}(z)]}_{x}}_{T_{3}}.
    \end{align*}
    To control $T_{1}$, we bound
    \begin{align*}
        T_1 \geq - \CVec \norm{g}_{z} d_{\mathcal{M}}(x,z)\norm{v - v_{z}}_x
        \geq -2\CVec\norm{g
        }_{z} \alpha^{2}
    \end{align*}
    where the first inequality follows from Cauchy-Schwarz, the isometry of parallel transport, and \eqref{eqn: transport error}. The second from triangle inequality and bounds on the distances and tangent vectors. 
    To control $T_{2}$, we bound
    \begin{align*}
        T_2 & \geq -\norm{g}_z \norm{v_{z} - \log_{x}(z)}_{x}\\
        & \geq - \norm{g}_z d_{\mathcal{M}}\left(\exp_{x}(v_{z}) , R_{x}(v_{z}) \right)\\
        & \geq - \CRet \norm{g}_z \norm{v_{z}}_{x}^{2} \\
        & \geq - \CRet \norm{g}_z \alpha^{2},
    \end{align*}
    where the first inequality follows from Cauchy-Schwarz and isometry of parallel transport, the second from exponential maps on Hadamard manifolds being distance-expanding, the third from \eqref{eqn: retraction error} with $x \in K$ and $\alpha \leq a$, and the last from $\|v_{z}\|_{x} \leq \alpha$. 
    Finally, \Cref{lem: parallel transport inherent error} upper bounds the magnitude of $T_{3}$. Combining these bounds with definition of $\kappa(\alpha,g)$ finishes the proof, after identifying $y = \exp_{x}(v)$.

\section{Proof of Lemma~\ref{lem: connect prox step to objective gap}}\label{subsec: proof of connect prox step to objective gap}
Let $f_\rho(w) = \min_{v \in T_w \cM}\{f(\exp_w(v)) + \frac{\rho}{2}\norm{v}^2\},$
which is well-defined since $f$ is $g$-convex. The following recovers \Cref{lem: connect prox step to objective gap}.
\begin{lemma} \label{lem: generalized textbook lemma for prox bound} Suppose that $\tilde{x} \in \cM$ is such that $f(\tilde{x}) < f(w)$. Then, 
\begin{equation*}
    f_\rho(w) \leq f(w) - \frac{\rho}{2}d_{\cM}^2(w,\tilde{x}) \varphi\biggr(\frac{f(w) - f(\tilde{x})}{\rho d_{\cM}^2(w, \tilde{x})}\biggr), \quad \varphi(\tau) = \begin{cases}
        \begin{matrix}
            0, & \textit{if } \tau < 0,\\
            \tau^2, & \textit{if } 0 \leq \tau \leq 1,\\
            -1 + 2 \tau, & \textit{if } \tau > 1.
        \end{matrix}
    \end{cases}
\end{equation*}
\end{lemma}
\begin{proof}
    There exists $v \in T_w \cM$ satisfying $\exp_w(v) = \tilde{x}$ as $\cM$ is Hadamard. Thus,
    \begin{align*}
        f_\rho(w) & \leq \min_{0 \leq t \leq 1} \biggr\{f(\exp_w(tv)) + \frac{\rho t^2}{2} \norm{v}^2\biggr\}\\
        & \leq f(w) + \min_{0\leq t \leq 1}\biggr\{t(f(\exp_w(v)) - f(\exp_w(0_{w}))) + \frac{\rho t^2}{2}\norm{v}^2\biggr\} \\
        & = f(w) + \min_{0 \leq t \leq 1}\biggr\{t(f(\tilde{x}) - f(w)) + \frac{\rho t^2}{2}d_{\cM}^2(w, \tilde{x})\biggr\},
    \end{align*}
    where the second inequality follows from geodesic convexity of $f$. First-order optimality conditions characterize the minimizer by
    \begin{equation*}
        \hat{t} = \min\biggr\{1, \frac{f(w) - f(\tilde{x})}{\rho d_{\cM}^2(w, \tilde{x})}\biggr\}. 
    \end{equation*}
    Substituting this result into the preceding inequality completes the proof.
\end{proof}
Applying this with $w = x_{i}, \tilde{x} = \operatorname{Proj}_{\mathcal{X}_{\ast}}(x_{i})$,  and $\tau = \delta_{i} /(\rho_{i}D_{i}^{2})$ recovers \Cref{lem: connect prox step to objective gap}.

\section{Proof of Lemma~\ref{lem: tuned proximal parameter sequence implication}} \label{appendix: proof of tuned proximal parameter sequence implication} First, by \eqref{eqn: specialized proximal parameter schedule}, $\rho_{k} \geq \mu^{2/p}\delta_{k}^{(p-2)/p}$,  equivalent to $\delta_{k}^{1-2/p} \leq \rho_{k}/\mu^{2/p}$. Second, as $\delta_{k}\searrow 0$, the term with the most negative exponent of $\delta_{k}$ becomes the most dominant in the schedule. The expression of $\swapGap$ is obtained by setting this term equal to each other term, solving for $\delta_{k}$, and taking the minimum. This term is precisely the stated form of $\rho_{k}$ when $\delta_{k} \leq \swapGap$. 

Finally, we bound descent and null steps until $\swapGap$ by bounding for each case of the schedule, then summing up the bounds for non-active schedule terms when $\delta_{k} \leq \swapGap$. 
Note that $\rho_{k} = \mu^{2/p}\delta_{k}^{(p-2)/p}$ is the same schedule as Theorem 2.6 of \cite{díaz2023optimal} and analogous arguments recover the log term in \eqref{eqn: bound on descent steps initial stage} and the second in \eqref{eqn: bound on null steps initial stage}. 

For every other case, combining  Lemma~\ref{lem: connect prox step to objective gap}, $p$-H\"{o}lder growth, the fact $\rho_{k}$ is constant at proximal centers by Lemma~\ref{lem: stepsize controls descent for recurrence}, and $\delta_{k}^{1-2/p}\leq \rho_{k}/\mu^{2/p}$ yields $\Delta_{k} \geq \mu^{2/p}\delta_{k}^{2-2/p}/2\rho_{k}$. Then, Lemma~\ref{lem: null step bound imperfect info} bounds the number of consecutive null steps at proximal center $x_{k}$
\begin{equation*}
        \frac{32 \constModelLip{\rho_{0}}^{2} G_{f}^{2}}{(1-\beta)^{2}\rho_{k}\Delta_{k}} \leq \frac{64 \constModelLip{\rho_{0}}^{2} G_{f}^{2}}{(1-\beta)^{2}\mu^{2/p}\delta_{k}^{2-2/p}} \leq \frac{64 \constModelLip{\rho_{0}}^{2} G_{f}^{2}}{(1-\beta)^{2}\mu^{2/p}\swapGap^{2-2/p}}
\end{equation*}
Multiplying the bound $T_{\operatorname{bound}}$ (to be derived), yields the corresponding terms in \eqref{eqn: bound on null steps initial stage}.

Lastly, Lemma~\ref{lem: descent guarantee} with the lower bound on $\Delta_{k}$ ensures on descent steps
\begin{equation*}
    \delta_{k+1} \leq \delta_{k} - \frac{\beta\mu^{2/p}\delta_{k}^{2-2/p}}{2\rho_{k}} \leq \begin{cases}
        \delta_{k} - \frac{\beta\mu^{4/p}\swapGap^{(4p-4)/p}}{2A}, & \rho_{k} = A\mu^{-2/p}\delta_{k}^{(2-2p)/p},\\
        \delta_{k} - \frac{\beta\mu^{2/3p}\swapGap^{(8p-8)/3p}}{2A^{1/3}\fLip^{1/3}}, \quad & \rho_{k} = (A\fLip \mu^{-2/p})^{1/3} \delta_{k}^{(2-2p)/3p},\\
        \delta_{k} - \frac{\beta\mu^{2/p}\swapGap^{(5p-4)/p}}{2A^{1/2}}, & \rho_{k} = (A/\delta_{k})^{1/2},\\
        \delta_{k} - \frac{\beta\mu^{2/p}\swapGap^{(7p-8)/4p}}{2A^{1/4}G_{f}^{1/4}}, & \rho_{k} = (AG_{f}/\delta_{k})^{1/4},
    \end{cases}
\end{equation*}
where the second inequality follows form $\delta_{k} \geq \swapGap$ and exponents being greater than 0. Solving the recurrences for $\delta_{k} \leq \swapGap$ recovers the non-log terms in \eqref{eqn: bound on descent steps initial stage}.

\end{document}